\documentclass[12pt,a4paper,reqno]{amsart}
\usepackage{amsmath}
\usepackage{amsthm}
\usepackage{amsfonts}
\usepackage{amssymb}
\usepackage{color}
\usepackage{mathrsfs}
\usepackage{graphicx}

\numberwithin{equation}{section}

\theoremstyle{plain}
\newtheorem{thm}{Theorem}

\theoremstyle{plain}
\newtheorem{lem}{Lemma}[section]

\theoremstyle{remark}
\newtheorem*{remark}{Remark}

\theoremstyle{definition}
\newtheorem*{claim1}{Claim 1}

\theoremstyle{definition}
\newtheorem*{claim2}{Claim 2}

\theoremstyle{definition}
\newtheorem*{claim3}{Claim 3}

\theoremstyle{definition}
\newtheorem*{claim4}{Claim 4}

\def\Qq{\mathbb{Q}}

\def\Zz{\mathbb{Z}}

\def\LLL{\mathcal{L}}

\def\PPP{\mathcal{P}}

\DeclareMathOperator{\mg}{MaxGap}

\renewcommand{\le}{\leqslant}
\renewcommand{\ge}{\geqslant}

\title[Time-quantitative density of non-integrable systems]
{Time-quantitative density\\
of non-integrable systems}

\author[Beck]{J. Beck}

\address{Department of Mathematics, Hill Center for the Mathematical Sciences,
Rutgers University, Piscataway NJ 08854, USA}

\email{jbeck@math.rutgers.edu}

\author[Chen]{W.W.L. Chen}

\address{School of Mathematical and Physical Sciences, Faculty of Science and Engineering,
Macquarie University, Sydney NSW 2109, Australia}

\email{william.chen@mq.edu.au}

\keywords{geodesics, billiards, density}

\subjclass[2010]{11K38, 37E35}

\begin{document}

\begin{abstract}
We introduce a new method to establish time-quantitative density in flat dynamical systems.
First we give a shorter and different proof of our earlier result in~\cite{BC}
that a half-infinite geodesic on an arbitrary finite polysquare surface $\PPP$ is superdense on $\PPP$
if the slope of the geodesic is a badly approximable number.
We then adapt our method to study time-quantitative density of half-infinite geodesics on algebraic polyrectangle surfaces.
\end{abstract}

\maketitle

\thispagestyle{empty}

%
%

\section{Introduction}\label{sec1}

A finite polysquare region $P$ is an arbitrary connected, but not necessarily simply-connected, polygon on the plane
which is tiled with closed unit squares, called the \textit{atomic squares} or \textit{square faces} of~$P$,
and which satisfies the following conditions:

(i) Any two atomic squares in $P$ either are disjoint, or intersect at a single point, or have a common edge.

(ii) Any two atomic squares in $P$ are joined by a chain of atomic squares where any two neighbors in the chain have a common edge.

Note that $P$ may have \textit{holes}, and we also allow \textit{whole barriers} which are horizontal or vertical \textit{walls}
that consist of one or more boundary edges of atomic squares.

Given such a finite polysquare region~$P$, we can convert it into a finite polysquare surface $\PPP$
by identification in pairs of the horizontal edges and identification in pairs of the vertical edges, as illustrated in Figure~1.1.
We can then consider $1$-direction geodesic flow on such a polysquare surface.

\begin{displaymath}
\begin{array}{c}
\includegraphics{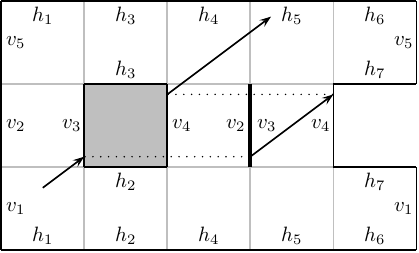}
\\
\mbox{Figure 1.1: a finite polysquare surface and part of a geodesic}
\end{array}
\end{displaymath}

A half-infinite $1$-direction geodesic $\LLL(t)$, $t\ge0$, on a given finite polysquare surface $\PPP$
and equipped with arc-length parametrization is superdense in $\PPP$ if there exists an absolute constant $C_1=C_1(\PPP;\LLL)>0$
such that, for every integer $n\ge1$, the initial segment $\LLL(t)$, $0\le t\le C_1n$, of the geodesic gets $1/n$-close to every point of ~$\PPP$.
This concept of \textit{superdensity}, which we first studied in~\cite{BCY1}, is a best possible form of time-quantitative density,
in the sense that the linear length $C_1n$ cannot be replaced by any sublinear length $o(n)$ as $n\to\infty$.
For a proof of this; see \cite[Section~6.1]{BCY1}.

In an earlier paper, we can establish the following result; see \cite[Theorem~1]{BC}.

\begin{thm}\label{thm1}
Let $\PPP$ be an arbitrary finite polysquare surface.
A half-infinite geodesic is superdense on $\PPP$ if and only if the slope of the geodesic is a badly approximable number.
\end{thm}

Theorem~\ref{thm1} is an \textit{if and only if} type result, where one of the two implications is a straightforward corollary of Khinchin's theorem.
Indeed, a $1$-direction geodesic flow on a finite polysquare surface \textit{modulo one} becomes a torus line flow on $[0,1)^2$,
and Khinchin's theorem then implies that a superdense geodesic must have a badly approximable slope.
The much harder task is to prove the converse, that every badly approximable slope leads to superdensity.

In Section~\ref{sec2}, we give a shorter and different proof of this result.
Whereas our earlier technique in~\cite{BC} works for finite polysquare surfaces,
it does not seem possible to extend it to study $1$-direction geodesics on more general surfaces.
Our new method here, on the other hand, is conducive to generalization,
and we shall discuss its adaptation to algebraic polyrectangle surfaces in Sections \ref{sec3} and~\ref{sec4}.

We remark also that a consequence of Theorem~\ref{thm1} is the corresponding result that a billiard orbit in a finite polysquare region
is superdense in the region if and only if the initial slope of the orbit is a badly approximable number.
This follows from a technique called unfolding.
For more details, see our earlier paper~\cite{BC}.

%
%

\section{Illustration of the method in the simplest case}\label{sec2}

We shall make use of an important property of badly approximable numbers.

\begin{lem}\label{lem21}
Suppose that $\alpha\in(0,1)$ is badly approximable, with continued fraction
\begin{displaymath}
\alpha=[a_1,a_2,a_3,\ldots]=\frac{1}{a_1+\frac{1}{a_2+\frac{1}{a_3+\cdots}}}.
\end{displaymath}
Suppose further that $A$ is a positive number such that the continued fraction digits $a_i\le A$, $i=1,2,3,\ldots.$
Then for every integer $n\ge1$, we have
\begin{displaymath}
\Vert n\alpha\Vert>\frac{1}{(A+2)n},
\end{displaymath}
where $\Vert\beta\Vert$ denotes the distance of the real number $\beta$ from the nearest integer.
\end{lem}

\begin{proof}
For every integer $n\ge1$, we can find an integer $i\ge0$ such that $q_i\le n<q_{i+1}$, where $q_i=q_i(\alpha)$
denotes the denominator of the $i$-th convergent of~$\alpha$.
Using well known diophantine approximation properties of continued fractions, we have
\begin{displaymath}
\Vert n\alpha\Vert
\ge\Vert q_i\alpha\Vert
\ge\frac{1}{q_i+q_{i+1}}
>\frac{1}{q_i+(a_{i+1}+1)q_i}
\ge\frac{1}{(A+2)q_i}
\ge\frac{1}{(A+2)n},
\end{displaymath}
as required.
\end{proof}

Suppose that an integer $i$ satisfies $1\le i\le s$, where $s$ is the number of atomic squares of the polysquare surface~$\PPP$.
We denote by $w_i$ the left vertical edge of the $i$-th atomic square of~$\PPP$, and by $w_i(0)$ and $w_i(1)$
the bottom and top endpoint of $w_i$ respectively, and in general by $w_i(q)$ the point on $w_i$ which is a distance $q$ from $w(0)$.
Furthermore, for any set $S\subseteq[0,1]$, we write
\begin{displaymath}
w_iS=\{w_i(q):q\in S\},
\end{displaymath}
so that $w_i=w_i[0,1]$.

Consider the geodesic $\LLL(t)=\LLL_\alpha(t)$ with slope $\alpha$ and starting point $\LLL(0)=R$,
where $R$ lies on the left vertical edge of the $i_0$-th atomic square of the polysquare surface~$\PPP$.
Assume that $\LLL(t)$ has arc-length parametrization, and that it does not hit a vertex of $\PPP$
over a sufficiently long neighborhood $-T\le t\le T$ of~$0$.
Then $R=w_{i_0}(y)$ for some $y$ satisfying $0<y<1$.
Let $Q=w_{i_0}(z)$ be a point where $0<z<y<1$.
We study the following question.
What can we say about the time $t$ with $\LLL(t)\in QR$ such that $\vert t\vert$ is minimum?
Here $QR$ denotes the open interval with endpoints $Q$ and~$R$.
In other words, how long does it take for the geodesic, starting at the point~$R$, to visit the open interval $QR$ of length $x=y-z$,
if the geodesic can go both forward and backward?

\begin{lem}\label{lem22}
Let $QR$ be an open vertical segment, with top endpoint $R$ and length $x>0$, on the left vertical edge of an atomic square
of the polysquare surface~$\PPP$.
Consider a geodesic $\LLL(t)$ with badly approximable slope $\alpha$ and starting point $\LLL(0)=R$.
There exists an explicit constant $c_0(A;s)$, depending at most on the parameter $A=A(\alpha)$
and the number $s$ of atomic squares of the polysquare surface~$\PPP$, such that there is a $2$-direction visiting time $t^*$ satisfying
\begin{displaymath}
0<\vert t^*\vert\le\frac{c_0(A;s)}{x}
\quad\mbox{and}\quad
\LLL(t^*)\in QR.
\end{displaymath}
\end{lem}

\begin{proof}
Let $S\subset[0,1]$ denote an open interval on the left vertical edge of an atomic square of~$\PPP$.
The $\alpha$-flow shifts $S$ until it hits some vertical edge or edges of $\PPP$ for the first time, with image $S(\alpha)$, say.
If $S(\alpha)$ contains a vertex of~$\PPP$, as in the picture on the right in Figure~2.1, then we say that the shift of $S$ by the $\alpha$-flow splits.
If $S(\alpha)$ does not contain a vertex of~$\PPP$, as in the picture on the left in Figure~2.1,
then we say that the shift of $S$ by the $\alpha$-flow does not split.

\begin{displaymath}
\begin{array}{c}
\includegraphics{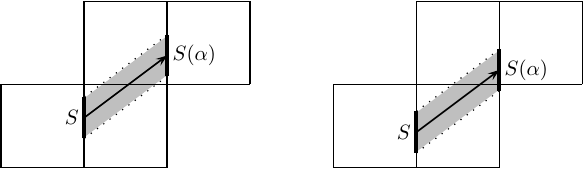}
\\
\mbox{Figure 2.1: shift of an interval by the $\alpha$-flow}
\end{array}
\end{displaymath}

Suppose now that $Q=w_{i_0}(z)$ and $R=w_{i_0}(y)$ where $0<z<y<1$.

If the shift of the open interval $QR$ under the $\alpha$-flow does not split, then there exists an integer $i_1$ satisfying $1\le i_1\le s$
such that $QR$ is shifted to an open interval $Q_1R_1$ on the vertical edge~$w_{i_1}$, and
\begin{displaymath}
Q_1=w_{i_1}(\{z+\alpha\})
\quad\mbox{and}\quad
R_1=w_{i_1}(\{y+\alpha\}),
\end{displaymath}
where $\{\beta\}$ denotes the fractional part of the real number~$\beta$.
Let us now repeat the argument with the open interval~$Q_1R_1$.
If the shift of $Q_1R_1$ under the $\alpha$-flow does not split, then there exists an integer $i_2$ satisfying $1\le i_2\le s$
such that $Q_1R_1$ is shifted to an open interval $Q_2R_2$ on the vertical edge~$w_{i_2}$, and
\begin{displaymath}
Q_2=w_{i_2}(\{z+2\alpha\})
\quad\mbox{and}\quad
R_2=w_{i_2}(\{y+2\alpha\}).
\end{displaymath}
We now repeat the argument with the open interval~$Q_2R_2$, and so on, until we get the first split.

\begin{claim1}
Suppose that there is no split among the first $[9s^2(A+2)/x]$ consecutive shifts of the open interval $QR$ under the $\alpha$-flow,
where $QR$ has length $x>0$.
Then there is a visiting time $t^*$ such that $0<t^*\le9\sqrt{2}s^2(A+2)/x$ and $\LLL(t^*)\in QR$,
so that the conclusion of Lemma~\ref{lem22} holds with a suitable constant $c_0(A;s)$.
\end{claim1}

\begin{proof}[Justification of Claim~1]
The open intervals $Q_jR_j$, where $1\le j\le9s^2(A+2)/x$, all have length~$x$.
Their total length is therefore at least $9s^2(A+2)-1$.
It follows easily from the Pigeonhole Principle that there exists a point $P$ which is covered by at least $8s(A+2)$ of these open intervals~$Q_jR_j$.
In other words, there exist integers $j_\nu$, $\nu=1,\ldots,8s(A+2)$, such that
\begin{displaymath}
1\le j_1<j_2<\ldots<j_{8s(A+2)}\le\frac{9s^2(A+2)}{x}
\end{displaymath}
such that $P$ is contained in
\begin{displaymath}
Q_{j_\nu}R_{j_\nu},
\quad
\nu=1,\ldots,8s(A+2).
\end{displaymath}
For every $\nu=1,\ldots,8s(A+2)$, the open interval $Q_{j_\nu}R_{j_\nu}$ lies on the vertical edge~$w_{i_{j_\nu}}$.
It follows that the values $i_{j_\nu}$, $\nu=1,\ldots,8s(A+2)$, are all equal to each other.
Suppose that $i^*$ is their common value.
Then for every $\nu=1,\ldots,8s(A+2)$, we can write
\textcolor{white}{xxxxxxxxxxxxxxxxxxxxxxxxxxxxxx}
\begin{displaymath}
R_{j_\nu}=w_{i^*}(u_\nu),
\end{displaymath}
where $0<u_\nu<1$, and write
\textcolor{white}{xxxxxxxxxxxxxxxxxxxxxxxxxxxxxx}
\begin{displaymath}
t_\nu=\sqrt{1+\alpha^2}j_\nu.
\end{displaymath}

Suppose first that there exist two integers $\nu'$ and $\nu''$ such that
\begin{equation}\label{eq2.1}
1\le\nu'<\nu''\le8s(A+2)
\quad\mbox{and}\quad
u_{\nu'}>u_{\nu''}.
\end{equation}
Since $Q_{j_{\nu'}}R_{j_{\nu'}}$ and $Q_{j_{\nu''}}R_{j_{\nu''}}$ intersect, we clearly have
\begin{displaymath}
R_{j_{\nu''}}=\LLL(t_{\nu''})\in Q_{j_{\nu'}}R_{j_{\nu'}}.
\end{displaymath}
Applying the reverse $\alpha$-flow for time~$t_{\nu'}$ then takes $Q_{j_{\nu'}}R_{j_{\nu'}}$ to~$QR$,
and also takes $R_{j_{\nu''}}=\LLL(t_{\nu''})$ to $\LLL(t_{\nu''}-t_{\nu'})$, so that $\LLL(t_{\nu''}-t_{\nu'})\in QR$.
Now take $t^*=t_{\nu''}-t_{\nu'}>0$.
Then
\textcolor{white}{xxxxxxxxxxxxxxxxxxxxxxxxxxxxxx}
\begin{displaymath}
0<t^*=\sqrt{1+\alpha^2}(j_{\nu''}-j_{\nu'})\le\frac{9\sqrt{2}s^2(A+2)}{x},
\end{displaymath}
justifying the claim.

Suppose next that there do not exist two integers $\nu'$ and $\nu''$ such that \eqref{eq2.1} holds.
Then we must have
\begin{displaymath}
u_1<u_2<\ldots<u_{8s(A+2)}
\quad\mbox{and}\quad
u_{8s(A+2)}-u_1\le x,
\end{displaymath}
and a routine average computation argument shows that for at least $5s(A+2)$ of the indices $\nu=1,\ldots,8s(A+2)$, we have
\begin{equation}\label{eq2.2}
u_{\nu+1}-u_\nu\le\frac{x}{3s(A+2)}.
\end{equation}
On the other hand, we also have
\begin{displaymath}
j_1<j_2<\ldots<j_{8s(A+2)}
\quad\mbox{and}\quad
j_{8s(A+2)}-j_1\le\frac{9s^2(A+2)}{x},
\end{displaymath}
and a routine average computation argument shows that for at least $5s(A+2)$ of the indices $\nu=1,\ldots,8s(A+2)$, we have
\begin{equation}\label{eq2.3}
j_{\nu+1}-j_\nu\le\frac{3s}{x}.
\end{equation}
It follows that there must exist some index $\nu=1,\ldots,8s(A+2)$ such that both \eqref{eq2.2} and \eqref{eq2.3} hold.
For this value of~$\nu$, Lemma~\ref{lem21} and \eqref{eq2.3} then lead to
\begin{equation}\label{eq2.4}
\Vert(j_{\nu+1}-j_\nu)\alpha\Vert
>\frac{1}{(A+2)(j_{\nu+1}-j_\nu)}
\ge\frac{x}{3s(A+2)}.
\end{equation}
On the other hand, in view of \eqref{eq2.2}, we have
\begin{equation}\label{eq2.5}
\Vert(j_{\nu+1}-j_\nu)\alpha\Vert=u_{\nu+1}-u_\nu\le\frac{x}{3s(A+2)}.
\end{equation}
Clearly \eqref{eq2.4} and \eqref{eq2.5} contradict each other, so this possibility cannot take place.

It follows that there exist two integers $\nu'$ and $\nu''$ such that \eqref{eq2.1} holds, and this completes our justification of Claim~1.
\end{proof}

In view of Claim~1, we may assume that there exists an integer $k$ such that
\begin{displaymath}
1\le k\le\frac{9s^2(A+2)}{x}
\end{displaymath}
and the $k$-th shift under the $\alpha$-flow of the open interval $QR$ of length $x$ splits for the first time.

Suppose that the image of the original open interval $QR$ after the first $k$ shifts under the $\alpha$-flow
now consists of a vertex of $\PPP$ and two intervals
\begin{displaymath}
w_{j_1}(0,x_1)
\quad\mbox{and}\quad
w_{\ell_1}(1-x_1^*,1),
\end{displaymath}
where $x_1+x_1^*=x$.
We call $w_{j_1}(0,x_1)$ and $w_{\ell_1}(1-x_1^*,1)$ respectively the \textit{top interval} and the \textit{bottom interval}.
Since the starting point $R$ of the geodesic is the \textit{top endpoint} of the interval $QR=w_{i_0}(z,y)$,
we shall make use of top intervals in our subsequent argument.
We distinguish two cases.
Either
\begin{equation}\label{eq2.6}
x_1\ge c_1(A;s)x
\quad\mbox{or}\quad
0<x_1<c_1(A;s)x,
\end{equation}
where the choice of the constant
\begin{equation}\label{eq2.7}
c_1(A;s)=\frac{1}{(36s^2(A+2)^2)^{s+1}}
\end{equation}
will be explained later.

Suppose that the first case in \eqref{eq2.6} holds.
Then we delete the bottom interval $w_{\ell_1}(1-x_1^*,1)$, keep the top interval $w_{j_1}(0,x_1)$ and write $Q^{(1)}R^{(1)}=w_{j_1}(0,x_1)$.
It then follows from our construction that the geodesic $\LLL(t)$, $t\ge0$, starting at the point~$R$, contains the point~$R^{(1)}$, so that
\begin{displaymath}
R^{(1)}=\LLL(t_1)
\quad\mbox{for some $t_1>0$}.
\end{displaymath}

We now repeat this argument on the open interval $Q^{(1)}R^{(1)}=w_{j_1}(0,x_1)$.

Corresponding to Claim~1, we have the following analog.
Suppose that there is no split among the first $[9s^2(A+2)/x_1]$ consecutive shifts of the open interval $Q^{(1)}R^{(1)}$
under the $\alpha$-flow, where $Q^{(1)}R^{(1)}$ has length $x_1>0$.
Let $\LLL^{(1)}(t)=\LLL(t+t_1)$ for every $t\ge0$.
Then there is a first visiting time $t^*$ such that $0<t^*\le9\sqrt{2}s^2(A+2)/x_1$ and $\LLL^{(1)}(t^*)\in Q^{(1)}R^{(1)}$,
\textit{i.e.}, $\LLL(t^*+t_1)\in Q^{(1)}R^{(1)}$.
Applying the reverse $\alpha$-flow for time $t_1$ then leads to $\LLL(t^*)\in QR$, so that the conclusion of Lemma~\ref{lem22}
holds with a suitable constant $c_0(A;s)$.
We may assume that there exists an integer $k_1$ such that
\begin{displaymath}
1\le k_1\le\frac{9s^2(A+2)}{x_1}
\end{displaymath}
and the $k_1$-th shift under the $\alpha$-flow of the open interval $Q^{(1)}R^{(1)}$ of length $x_1$ splits for the first time.

Suppose that the image of the open interval $Q^{(1)}R^{(1)}$ after the first $k_1$ shifts under the $\alpha$-flow
now consists of a vertex of $\PPP$ and two intervals
\begin{displaymath}
w_{j_2}(0,x_2)
\quad\mbox{and}\quad
w_{\ell_2}(1-x_2^*,1),
\end{displaymath}
where $x_2+x_2^*=x_1$.
We distinguish two cases.
Either
\begin{equation}\label{eq2.8}
x_2\ge c_1(A;s)x_1
\quad\mbox{or}\quad
0<x_2<c_1(A;s)x_1,
\end{equation}
where the constant $c_1(A;s)$ is defined by \eqref{eq2.7}.

Suppose that the first case in \eqref{eq2.8} holds.
Then we delete the bottom interval $w_{\ell_2}(1-x_2^*,1)$, keep the top interval $w_{j_2}(0,x_2)$ and write $Q^{(2)}R^{(2)}=w_{j_2}(0,x_2)$.
It then follows from our construction that the geodesic $\LLL(t)$, $t\ge0$, starting at the point~$R$, contains the point~$R^{(2)}$, so that
\begin{displaymath}
R^{(2)}=\LLL(t_2)
\quad\mbox{for some $t_2>t_1$}.
\end{displaymath}

We now repeat this argument on the open interval $Q^{(2)}R^{(2)}=w_{j_2}(0,x_2)$.

In view of another suitable analog of Claim~1, we may assume that there exists an integer $k_2$ such that
\textcolor{white}{xxxxxxxxxxxxxxxxxxxxxxxxxxxxxx}
\begin{displaymath}
1\le k_2\le\frac{9s^2(A+2)}{x_2}
\end{displaymath}
and the $k_2$-th shift under the $\alpha$-flow of the open interval $Q^{(2)}R^{(2)}$ of length $x_2$ splits for the first time.

Suppose that the image of the open interval $Q^{(2)}R^{(2)}$ after the first $k_2$ shifts under the $\alpha$-flow now consists of a vertex of
$\PPP$ and two intervals
\begin{displaymath}
w_{j_3}(0,x_3)
\quad\mbox{and}\quad
w_{\ell_3}(1-x_3^*,1),
\end{displaymath}
where $x_3+x_3^*=x_2$.
We distinguish two cases.
Either
\begin{equation}\label{eq2.9}
x_3\ge c_1(A;s)x_2
\quad\mbox{or}\quad
0<x_3<c_1(A;s)x_2,
\end{equation}
where the constant $c_1(A;s)$ is defined by \eqref{eq2.7}.

Suppose that the first case in \eqref{eq2.9} holds.
Then we delete the bottom interval $w_{\ell_3}(1-x_3^*,1)$, keep the top interval $w_{j_3}(0,x_3)$ and write $Q^{(3)}R^{(3)}=w_{j_3}(0,x_3)$.
It then follows from our construction that the geodesic $\LLL(t)$, $t\ge0$, starting at the point~$R$, contains the point~$R^{(3)}$, so that
\begin{displaymath}
R^{(3)}=\LLL(t_3)
\quad\mbox{for some $t_3>t_2$}.
\end{displaymath}

We now repeat this argument on the open interval $Q^{(3)}R^{(3)}=w_{j_3}(0,x_3)$.

And so on, \textit{assuming that at each step, the first case in the corresponding analog of \eqref{eq2.6}, \eqref{eq2.8} and \eqref{eq2.9} holds}.

This forward shift process under the $\alpha$-flow defines a sequence of top intervals
\begin{equation}\label{eq2.10}
Q^{(i)}R^{(i)}=w_{j_i}(0,x_i),
\quad
i\ge1,
\end{equation}
each of which arises when the $k_{i-1}$-th shift under the $\alpha$-flow of the open interval $Q^{(i-1)}R^{(i-1)}$ of length $x_{i-1}$
splits for the first time, and the integer $k_{i-1}$ satisfies
\begin{equation}\label{eq2.11}
1\le k_{i-1}\le\frac{9s^2(A+2)}{x_{i-1}}.
\end{equation}
The lengths $x_i$ of these intervals \eqref{eq2.10} satisfy
\begin{equation}\label{eq2.12}
x_i\ge c_1(A;s)x_{i-1},
\end{equation}
with the convention that $x_0=x$.
Furthermore, the geodesic $\LLL(t)$, $t\ge0$, starting at the point $R$ contains the point~$R^{(i)}$, so that
\begin{displaymath}
R^{(i)}=\LLL(t_i)
\quad\mbox{for some $t_i>t_{i-1}$},
\end{displaymath}
where $t_0=0$.

As there are only finitely many vertical edges in the polysquare surface~$\PPP$, there will at some point be \textit{edge repetition},
when there exist two integers $i_1$ and $i_2$ satisfying $1\le i_1<i_2$ such that the corresponding top intervals
\begin{displaymath}
Q^{(i_1)}R^{(i_1)}=w_{j_{i_1}}(0,x_{i_1})
\quad\mbox{and}\quad
Q^{(i_2)}R^{(i_2)}=w_{j_{i_2}}(0,x_{i_2})
\end{displaymath}
lying respectively on the vertical edges $w_{j_{i_1}}$ and~$w_{j_{i_2}}$, overlap.
Thus $j_{i_1}=j_{i_2}$.
Now suppose that $j^*$ is their common value.
Then
\begin{equation}\label{eq2.13}
Q^{(i_1)}R^{(i_1)}=w_{j^*}(0,x_{i_1})
\quad\mbox{and}\quad
Q^{(i_2)}R^{(i_2)}=w_{j^*}(0,x_{i_2}).
\end{equation}
Furthermore, since there are precisely $s$ vertical edges on~$\PPP$, it follows that
\begin{equation}\label{eq2.14}
1\le i_1<i_2\le s+1.
\end{equation}

\begin{claim2}
Suppose that there exist integers $i_1$ and $i_2$ satisfying \eqref{eq2.14} such that the following conditions hold:

(1)
For every integer $i$ satisfying $1\le i\le i_2$, there exists an integer $k_{i-1}$ satisfying \eqref{eq2.11}
such that the top interval $Q^{(i)}R^{(i)}$ given by \eqref{eq2.10} arises when the $k_{i-1}$-th shift
under the $\alpha$-flow of the open interval $Q^{(i-1)}R^{(i-1)}$ splits for the first time, where $Q^{(0)}R^{(0)}=QR$.

(2)
For every integer $i$ satisfying $1\le i\le i_2$, the condition \eqref{eq2.12} holds, where $x_0=x$.

(3)
There exists an integer $j^*$ such that the condition \eqref{eq2.13} holds.

Then there is a visiting time $t^*$ such that $0<t^*\le t_{i_2}$ and $\LLL(t^*)\in QR$, where $R^{(i_2)}=\LLL(t_{i_2})$,
and the conclusion of Lemma~\ref{lem22} holds with a suitable constant $c_0(A;s)$.
\end{claim2}

\begin{proof}[Justification of Claim~2]
Since $i_1<i_2$, we have $x_{i_1}>x_{i_2}$.
It follows from \eqref{eq2.13} that
\begin{displaymath}
R^{(i_2)}=\LLL(t_{i_2})\in Q^{(i_1)}R^{(i_1)}.
\end{displaymath}
Applying the reverse $\alpha$-flow for time $t_{i_1}$ then takes $Q^{(i_1)}R^{(i_1)}$ to~$QR$,
and also takes $R^{(i_2)}=\LLL(t_{i_2})$ to $\LLL(t_{i_2}-t_{i_1})$, so that $\LLL(t_{i_2}-t_{i_1})\in QR$.
This justifies the first assertion in Claim~2.
Next, note that the open interval $Q^{(i_2)}R^{(i_2)}$ arises as a consequence of
\begin{displaymath}
k+k_1+\ldots+k_{i_2-1}
\le\sum_{i=0}^s\frac{9s^2(A+2)}{x_i}
\le\frac{1}{x}\sum_{i=0}^s\frac{9s^2(A+2)}{(c_1(A;s))^i}
\end{displaymath}
consecutive shifts under the $\alpha$-flow of the open interval~$QR$, using \eqref{eq2.11} and \eqref{eq2.12}.
Since each shift under the $\alpha$-flow corresponds to a geodesic segment of length $\sqrt{1+\alpha^2}\le\sqrt{2}$, it follows that
\begin{displaymath}
t_{i_2}\le\frac{1}{x}\sum_{i=0}^s\frac{9\sqrt{2}s^2(A+2)}{(c_1(A;s))^i}.
\end{displaymath}
If the constant $c_0(A;s)$ in Lemma~\ref{lem22} is chosen to satisfy
\begin{equation}\label{eq2.15}
c_0(A;s)\ge\sum_{i=0}^s\frac{9\sqrt{2}s^2(A+2)}{(c_1(A;s))^i},
\end{equation}
then the conclusion of Lemma~\ref{lem22} holds.
\end{proof}

Suppose next that before edge repetition takes place, the condition \eqref{eq2.12} fails.
More precisely, suppose that $r$ satisfying $0\le r\le s$ is the smallest integer $i$ such that $x_{i+1}<c_1(A;s)x_i$.
Then
\textcolor{white}{xxxxxxxxxxxxxxxxxxxxxxxxxxxxxx}
\begin{equation}\label{eq2.16}
x_{r+1}<c_1(A;s)x_r,
\end{equation}
and
\textcolor{white}{xxxxxxxxxxxxxxxxxxxxxxxxxxxxxx}
\begin{equation}\label{eq2.17}
x_i\ge c_1(A;s)x_{i-1},
\quad
i=1,\ldots,r,
\end{equation}
with $x_0=x$.
Furthermore, there exists an integer $k_r$ such that
\begin{equation}\label{eq2.18}
1\le k_r\le\frac{9s^2(A+2)}{x_r}
\end{equation}
and the $k_r$-th shift under the $\alpha$-flow of the open interval $Q^{(r)}R^{(r)}$ of length $x_r$ splits for the first time,
with the image consisting of a vertex of $\PPP$ and two intervals
\begin{displaymath}
w_{j_{r+1}}(0,x_{r+1})
\quad\mbox{and}\quad
w_{\ell_{r+1}}(1-x_{r+1}^*,1),
\end{displaymath}
where $x_{r+1}+x_{r+1}^*=x_r$ and \eqref{eq2.16} holds.

We now start with the interval $Q^{(r)}R^{(r)}=w_{j_r}(0,x_r)$ and apply the reverse $\alpha$-flow until it splits for the first time on a vertical edge.
We have the following analog of Claim~1.

\begin{claim3}
Suppose that there is no split among the first $[9s^2(A+2)/x_r]$ consecutive shifts of the open interval $Q^{(r)}R^{(r)}$
under the reverse $\alpha$-flow, where $Q^{(r)}R^{(r)}$ has length $x_r>0$.
Let $\LLL_r(t)=\LLL(t+t_r)$ for every $t\le0$.
Then there is a visiting time $t^*$ such that $0>t^*\ge-9\sqrt{2}s^2(A+2)/x_r$ and $\LLL_r(t^*)\in Q^{(r)}R^{(r)}$, \textit{i.e.}
$\LLL(t^*+t_r)\in Q^{(r)}R^{(r)}$.
Applying the reverse $\alpha$-flow for time $t_r$ then leads to $\LLL(t^*)\in QR$, so that the conclusion of Lemma~\ref{lem22}
holds with a suitable constant $c_0(A;s)$.
\end{claim3}

\begin{proof}[Justification of Claim~3]
Let
\begin{displaymath}
Q^{(r)}_jR^{(r)}_j,
\quad
j=1,\ldots,[9s^2(A+2)/x_r],
\end{displaymath}
be successive open intervals under the reverse $\alpha$-flow, starting at $Q^{(r)}R^{(r)}$.
These intervals all have length~$x_r$.
Their total length is therefore at least $9s^2(A+2)-1$.
It follows easily from the Pigeonhole Principle that there exists a point $P$ which is covered by at least $8s(A+2)$ of these open
intervals~$Q^{(r)}_jR^{(r)}_j$.
In other words, there exist integers $j_\nu$, $\nu=1,\ldots,8s(A+2)$, such that
\begin{displaymath}
1\le j_1<j_2<\ldots<j_{8s(A+2)}\le\frac{9s^2(A+2)}{x_r}
\end{displaymath}
such that $P$ is contained in
\begin{displaymath}
Q^{(r)}_{j_\nu}R^{(r)}_{j_\nu},
\quad
\nu=1,\ldots,8s(A+2),
\end{displaymath}
and there exists $i^*$ such that all these intervals lie on the same vertical edge~$w_{i^*}$.
Then for every $\nu=1,\ldots,8s(A+2)$, we can write
\begin{displaymath}
R^{(r)}_{j_\nu}=w_{i^*}(u_\nu),
\end{displaymath}
where $0<u_\nu<1$, and write
\textcolor{white}{xxxxxxxxxxxxxxxxxxxxxxxxxxxxxx}
\begin{displaymath}
t_\nu=-\sqrt{1+\alpha^2}j_\nu.
\end{displaymath}

Suppose first that there exist two integers $\nu'$ and $\nu''$ such that
\begin{equation}\label{eq2.19}
1\le\nu'<\nu''\le8s(A+2)
\quad\mbox{and}\quad
u_{\nu'}>u_{\nu''}.
\end{equation}
Since $Q_{j_{\nu'}}R_{j_{\nu'}}$ and $Q_{j_{\nu''}}R_{j_{\nu''}}$ intersect, we clearly have
\begin{displaymath}
R^{(r)}_{j_{\nu''}}=\LLL_r(t_{\nu''})\in Q^{(r)}_{j_{\nu'}}R^{(r)}_{j_{\nu'}}.
\end{displaymath}
Applying the forward $\alpha$-flow for time~$-t_{\nu'}$ then takes $Q^{(r)}_{j_{\nu'}}R^{(r)}_{j_{\nu'}}$ to~$Q^{(r)}R^{(r)}$, and also takes
$R^{(r)}_{j_{\nu''}}=\LLL_r(t_{\nu''})$ to $\LLL_r(t_{\nu''}-t_{\nu'})$, so that $\LLL_r(t_{\nu''}-t_{\nu'})\in Q^{(r)}R^{(r)}$.
Now take $t^*=t_{\nu''}-t_{\nu'}<0$.
Then
\begin{displaymath}
0>t^*=-\sqrt{1+\alpha^2}(j_{\nu''}-j_{\nu'})\ge-\frac{9\sqrt{2}s^2(A+2)}{x_r}.
\end{displaymath}
Thus $\LLL(t^*+t_r)=\LLL_r(t^*)\in Q^{(r)}R^{(r)}$, justifying the claim.

Suppose next that there do not exist two integers $\nu'$ and $\nu''$ such that \eqref{eq2.19} holds.
Then we can show as in the justification of Claim~1 that this possibility cannot take place.
It follows that there exist two integers $\nu'$ and $\nu''$ such that \eqref{eq2.19} holds, and this completes our justification of Claim~3.
\end{proof}

\begin{remark}
Note that there is no split among the first $[9s^2(A+2)/x_r]$ consecutive shifts of the open interval $Q^{(r)}R^{(r)}$
under the reverse $\alpha$-flow in Claim~3, even if this takes us back to the original vertical interval $QR$ and beyond,
as this is our assumption.
On the other hand, the conclusion of Claim~3 that there exists some $t^*<0$ such that $\LLL(t^*)\in QR$
makes one wonder what happens to $QR$ and some of its subintervals under the influence of the reverse $\alpha$-flow.
This is, however, an unwelcome distraction.
The reality is that we have found this special $t^*$ by proper means, and the effect of this reverse $\alpha$-flow on $QR$
and some of its subintervals is totally irrelevant.
\end{remark}

In view of Claim~3, we may assume that there exists an integer $k_{r+1}$ such that
\begin{equation}\label{eq2.20}
1\le k_{r+1}\le\frac{9s^2(A+2)}{x_r}
\end{equation}
and the $k_{r+1}$-th shift under the reverse $\alpha$-flow of the open interval $Q^{(r)}R^{(r)}$ of length $x_r$ splits for the first time.

Suppose that the image of the open interval $Q^{(r)}R^{(r)}$ after the first $k_{r+1}$ shifts under the reverse $\alpha$-flow
now consists of a vertex of $\PPP$ and two intervals
\begin{displaymath}
w_{j'_{r+1}}(0,y_{r+1})
\quad\mbox{and}\quad
w_{\ell'_{r+1}}(1-y_{r+1}^*,1),
\end{displaymath}
where $y_{r+1}+y_{r+1}^*=x_r$.
Then we delete the bottom interval $w_{\ell'_{r+1}}(1-y_{r+1}^*,1)$, keep the top interval $w_{j'_{r+1}}(0,y_{r+1})$ and write
$Q_{r+1}R_{r+1}=w_{j'_{r+1}}(0,y_{r+1})$.
It then follows from our construction that the geodesic $\LLL(t)$ starting at the point $R$ contains the point~$R_{r+1}$, so that
\textcolor{white}{xxxxxxxxxxxxxxxxxxxxxxxxxxxxxx}
\begin{displaymath}
R_{r+1}=\LLL(t_{r+1})
\quad\mbox{for some $t_{r+1}$},
\end{displaymath}
where $t_{r+1}$ can be positive or negative.

To estimate $y_{r+1}$ from below, note that the point $R_{r+1}=w_{j'_{r+1}}(y_{r+1})$ is obtained from the point
$R^{(r+1)}=w_{j_{r+1}}(x_{r+1})$ by $k_r$ shifts under the reverse $\alpha$-flow to the point $R^{(r)}$
followed by another $k_{r+1}$ shifts under the reverse $\alpha$-flow from the point~$R^{(r)}$.
Each shift under the reverse $\alpha$-flow corresponds to a vertical descent of~$\alpha$, and so the total descent is $(k_r+k_{r+1})\alpha$.
It follows that
\begin{displaymath}
\{x_{r+1}-(k_r+k_{r+1})\alpha\}=y_{r+1},
\end{displaymath}
so there exists an integer $n_0$ such that
\begin{displaymath}
x_{r+1}-(k_r+k_{r+1})\alpha-n_0=y_{r+1}.
\end{displaymath}
Since $0<x_{r+1},y_{r+1}<1$, we then have
\begin{equation}\label{eq2.21}
\Vert(k_r+k_{r+1})\alpha\Vert=\Vert x_{r+1}-y_{r+1}\Vert\le x_{r+1}+y_{r+1},
\end{equation}
so that in view of Lemma~\ref{lem21}, \eqref{eq2.16}, \eqref{eq2.18} and \eqref{eq2.20}, we have
\begin{align}
y_{r+1}
&
\ge\Vert(k_r+k_{r+1})\alpha\Vert-x_{r+1}
\ge\frac{1}{(A+2)(k_r+k_{r+1})}-x_{r+1}
\nonumber
\\
&
\ge\left(\frac{1}{18s^2(A+2)^2}-c_1(A;s)\right)x_r.
\nonumber
\end{align}
It is clear from \eqref{eq2.7} that
\textcolor{white}{xxxxxxxxxxxxxxxxxxxxxxxxxxxxxx}
\begin{displaymath}
c_1(A;s)\le\frac{1}{36s^2(A+2)^2}.
\end{displaymath}
It follows that
\textcolor{white}{xxxxxxxxxxxxxxxxxxxxxxxxxxxxxx}
\begin{equation}\label{eq2.22}
y_{r+1}\ge\delta_1(A;s)x_r,
\end{equation}
where
\textcolor{white}{xxxxxxxxxxxxxxxxxxxxxxxxxxxxxx}
\begin{equation}\label{eq2.23}
\delta_1(A;s)=\frac{1}{36s^2(A+2)^2}.
\end{equation}

We now repeat this argument on the open interval $Q_{r+1}R_{r+1}=w_{j'_{r+1}}(0,y_{r+1})$.

In view of a suitable analog of Claim~3, we may assume that there exists an integer $k_{r+2}$ such that
\textcolor{white}{xxxxxxxxxxxxxxxxxxxxxxxxxxxxxx}
\begin{equation}\label{eq2.24}
1\le k_{r+2}\le\frac{9s^2(A+2)}{y_{r+1}}
\end{equation}
and the $k_{r+2}$-th shift under the reverse $\alpha$-flow of the open interval $Q_{r+1}R_{r+1}$ of length $y_{r+1}$ splits for the first time.

Suppose that the image of the open interval $Q_{r+1}R_{r+1}$ after the first $k_{r+2}$ shifts under the reverse $\alpha$-flow
now consists of a vertex of $\PPP$ and two intervals
\begin{displaymath}
w_{j'_{r+2}}(0,y_{r+2})
\quad\mbox{and}\quad
w_{\ell'_{r+2}}(1-y_{r+2}^*,1),
\end{displaymath}
where $y_{r+2}+y_{r+2}^*=y_{r+1}$.
Then we delete the bottom interval $w_{\ell'_{r+2}}(1-y_{r+2}^*,1)$, keep the top interval $w_{j'_{r+2}}(0,y_{r+2})$ and write
$Q_{r+2}R_{r+2}=w_{j'_{r+2}}(0,y_{r+2})$.
It then follows from our construction that the geodesic $\LLL(t)$ starting at the point $R$ contains the point~$R_{r+2}$, so that
\textcolor{white}{xxxxxxxxxxxxxxxxxxxxxxxxxxxxxx}
\begin{displaymath}
R_{r+2}=\LLL(t_{r+2})
\quad\mbox{for some $t_{r+2}$},
\end{displaymath}
where $t_{r+2}$ can be positive or negative.

To estimate $y_{r+2}$ from below, note that the point $R_{r+2}=w_{j'_{r+2}}(y_{r+2})$ is obtained from the point
$R^{(r+1)}=w_{j_{r+1}}(x_{r+1})$ by $k_r$ shifts under the reverse $\alpha$-flow to the point $R^{(r)}$
followed by another $k_{r+1}+k_{r+2}$ shifts under the reverse $\alpha$-flow from the point~$R^{(r)}$.
Each shift under the $\alpha$-flow corresponds to a vertical descent of~$\alpha$, and so the total descent is $(k_r+k_{r+1}+k_{r+2})\alpha$.
It follows that
\begin{displaymath}
\{x_{r+1}-(k_r+k_{r+1}+k_{r+2})\alpha\}=y_{r+2},
\end{displaymath}
so the analog of \eqref{eq2.21} is
\begin{equation}\label{eq2.25}
\Vert(k_r+k_{r+1}+k_{r+2})\alpha\Vert=\Vert x_{r+1}-y_{r+2}\Vert\le x_{r+1}+y_{r+2}.
\end{equation}
By \eqref{eq2.18}, \eqref{eq2.20}, \eqref{eq2.22}, \eqref{eq2.23} and \eqref{eq2.24}, we have
\begin{align}\label{eq2.26}
k_r+k_{r+1}+k_{r+2}
&
\le\frac{18s^2(A+2)}{x_r}+\frac{9s^2(A+2)}{y_{r+1}}
\nonumber
\\
&
\le\frac{18s^2(A+2)}{x_r}+\frac{324s^4(A+2)^3}{x_r}
\le\frac{648s^4(A+2)^3}{x_r}.
\end{align}
Combining \eqref{eq2.25} with Lemma~\ref{lem21}, \eqref{eq2.16} and \eqref{eq2.26} , we have
\begin{align}
y_{r+2}
&
\ge\Vert(k_r+k_{r+1}+k_{r+2})\alpha\Vert-x_{r+1}
\ge\frac{1}{(A+2)(k_r+k_{r+1}+k_{r+2})}-x_{r+1}
\nonumber
\\
&
\ge\left(\frac{1}{648s^4(A+2)^4}-c_1(A;s)\right)x_r.
\nonumber
\end{align}
It is clear from \eqref{eq2.7} that
\begin{displaymath}
c_1(A;s)\le\frac{1}{1296s^4(A+2)^4}=\frac{1}{(36s^2(A+2)^2)^2}.
\end{displaymath}
It follows that
\textcolor{white}{xxxxxxxxxxxxxxxxxxxxxxxxxxxxxx}
\begin{equation}\label{eq2.27}
y_{r+2}\ge\delta_2(A;s)x_r,
\end{equation}
where
\textcolor{white}{xxxxxxxxxxxxxxxxxxxxxxxxxxxxxx}
\begin{equation}\label{eq2.28}
\delta_2(A;s)=\frac{1}{1296s^4(A+2)^4}=\frac{1}{(36s^2(A+2)^2)^2}.
\end{equation}

We now repeat this argument on the open interval $Q_{r+2}R_{r+2}=w_{j'_{r+2}}(0,y_{r+2})$.

In view of a suitable analog of Claim~3, we may assume that there exists an integer $k_{r+3}$ such that
\textcolor{white}{xxxxxxxxxxxxxxxxxxxxxxxxxxxxxx}
\begin{equation}\label{eq2.29}
1\le k_{r+3}\le\frac{9s^2(A+2)}{y_{r+2}}
\end{equation}
and the $k_{r+3}$-th shift under the reverse $\alpha$-flow of the open interval $Q_{r+2}R_{r+2}$ of length $y_{r+2}$ splits for the first time.

Suppose that the image of the open interval $Q_{r+2}R_{r+2}$ after the first $k_{r+3}$ shifts under the reverse $\alpha$-flow
now consists of a vertex of $\PPP$ and two intervals
\begin{displaymath}
w_{j'_{r+3}}(0,y_{r+3})
\quad\mbox{and}\quad
w_{\ell'_{r+3}}(1-y_{r+3}^*,1),
\end{displaymath}
where $y_{r+3}+y_{r+3}^*=y_{r+2}$.
Then we delete the bottom interval $w_{\ell'_{r+3}}(1-y_{r+3}^*,1)$, keep the top interval $w_{j'_{r+3}}(0,y_{r+3})$ and write
$Q_{r+3}R_{r+3}=w_{j'_{r+3}}(0,y_{r+3})$.
It then follows from our construction that the geodesic $\LLL(t)$ starting at the point $R$ contains the point~$R_{r+3}$, so that
\textcolor{white}{xxxxxxxxxxxxxxxxxxxxxxxxxxxxxx}
\begin{displaymath}
R_{r+3}=\LLL(t_{r+3})
\quad\mbox{for some $t_{r+3}$},
\end{displaymath}
where $t_{r+3}$ can be positive or negative.

To estimate $y_{r+3}$ from below, note that the point $R_{r+3}=w_{j'_{r+3}}(y_{r+3})$ is obtained from the point
$R^{(r+1)}=w_{j_{r+1}}(x_{r+1})$ by $k_r$ shifts under the reverse $\alpha$-flow to the point $R^{(r)}$
followed by another $k_{r+1}+k_{r+2}+k_{r+3}$ shifts under the reverse $\alpha$-flow from the point~$R^{(r)}$.
Each shift under the $\alpha$-flow corresponds to a vertical descent of~$\alpha$,
and so the total descent is $(k_r+k_{r+1}+k_{r+2}+k_{r+3})\alpha$.
It follows that
\begin{displaymath}
\{x_{r+1}-(k_r+k_{r+1}+k_{r+2}+k_{r+3})\alpha\}=y_{r+3},
\end{displaymath}
so the analog of \eqref{eq2.21} and \eqref{eq2.25} is
\begin{equation}\label{eq2.30}
\Vert(k_r+k_{r+1}+k_{r+2}+k_{r+3})\alpha\Vert=\Vert x_{r+1}-y_{r+3}\Vert\le x_{r+1}+y_{r+3}.
\end{equation}
By \eqref{eq2.26}--\eqref{eq2.29}, we have
\begin{align}\label{eq2.31}
k_r+k_{r+1}+k_{r+2}+k_{r+3}
&
\le\frac{648s^4(A+2)^3}{x_r}+\frac{9s^2(A+2)}{y_{r+2}}
\nonumber
\\
&
\le\frac{648s^4(A+2)^3}{x_r}+\frac{11664s^6(A+2)^5}{x_r}
\nonumber
\\
&
\le\frac{23328s^6(A+2)^5}{x_r}.
\end{align}
Combining \eqref{eq2.30} with Lemma~\ref{lem21}, \eqref{eq2.16} and \eqref{eq2.31} , we have
\begin{align}
y_{r+3}
&
\ge\Vert(k_r+k_{r+1}+k_{r+2}+k_{r+3})\alpha\Vert-x_{r+1}
\nonumber
\\
&
\ge\frac{1}{(A+2)(k_r+k_{r+1}+k_{r+2}+k_{r+3})}-x_{r+1}
\nonumber
\\
&
\ge\left(\frac{1}{23328s^6(A+2)^6}-c_1(A;s)\right)x_r.
\nonumber
\end{align}
It is clear from \eqref{eq2.7} that
\begin{displaymath}
c_1(A;s)\le\frac{1}{46656s^6(A+2)^6}=\frac{1}{(36s^2(A+2)^2)^3}.
\end{displaymath}
It follows that
\textcolor{white}{xxxxxxxxxxxxxxxxxxxxxxxxxxxxxx}
\begin{equation}\label{eq2.32}
y_{r+3}\ge\delta_3(A;s)x_r,
\end{equation}
where
\textcolor{white}{xxxxxxxxxxxxxxxxxxxxxxxxxxxxxx}
\begin{equation}\label{eq2.33}
\delta_3(A;s)=\frac{1}{46656s^6(A+2)^6}=\frac{1}{(36s^2(A+2)^2)^3}.
\end{equation}

We now repeat this argument on the open interval $Q_{r+3}R_{r+3}=w_{j'_{r+3}}(0,y_{r+3})$.

And so on.

This shift process under the reverse $\alpha$-flow defines a sequence of top intervals
\begin{equation}\label{eq2.34}
Q_{r+i}R_{r+i}=w_{j'_{r+i}}(0,y_{r+i}),
\quad
i\ge1,
\end{equation}
with length
\textcolor{white}{xxxxxxxxxxxxxxxxxxxxxxxxxxxxxx}
\begin{equation}\label{eq2.35}
y_{r+i}\ge\delta_i(A;s)x_r,
\end{equation}
where
\textcolor{white}{xxxxxxxxxxxxxxxxxxxxxxxxxxxxxx}
\begin{equation}\label{eq2.36}
\delta_i(A;s)=\frac{1}{(36s^2(A+2)^2)^i},
\end{equation}
as long as we ensure that
\textcolor{white}{xxxxxxxxxxxxxxxxxxxxxxxxxxxxxx}
\begin{equation}\label{eq2.37}
c_1(A;s)\le\frac{1}{(36s^2(A+2)^2)^i}.
\end{equation}
Each interval in \eqref{eq2.34} arises when the $k_{r+i}$-th shift under the reverse $\alpha$-flow
of the open interval $Q_{r+i-1}R_{r+i-1}$ of length $y_{r+i-1}$ splits for the first time, and the integer $k_{r+i}$ satisfies
\textcolor{white}{xxxxxxxxxxxxxxxxxxxxxxxxxxxxxx}
\begin{equation}\label{eq2.38}
1\le k_{r+i}\le\frac{9s^2(A+2)}{y_{r+i-1}},
\end{equation}
with the convention that $y_r=x_r$.
Furthermore, the geodesic $\LLL(t)$ starting at the point $R$ contains the point~$R_{r+i}$, so that
\begin{displaymath}
R_{r+i}=\LLL(t_{r+i})
\quad\mbox{for some $t_{r+i}$},
\end{displaymath}
where $t_{r+i}$ may be positive or negative.

\begin{remark}
The assertions \eqref{eq2.35} and \eqref{eq2.36} can be proved easily by induction on the parameter~$i$.
For the initial cases $i=1,2,3$, see \eqref{eq2.22}, \eqref{eq2.23}, \eqref{eq2.27}, \eqref{eq2.28}, \eqref{eq2.32} and \eqref{eq2.33}.
\end{remark}

As there are only finitely many vertical edges in the polysquare surface~$\PPP$, there will at some point be \textit{edge repetition},
when there exist two integers $i_1$ and $i_2$ satisfying $1\le i_i<i_2$ such that the corresponding top intervals
\begin{displaymath}
Q_{r+i_1}R_{r+i_1}=w_{j'_{r+i_1}}(0,y_{r+i_1})
\quad\mbox{and}\quad
Q_{r+i_2}R_{r+i_2}=w_{j'_{r+i_2}}(0,y_{r+i_2})
\end{displaymath}
lying respectively on the vertical edges $w_{j'_{r+i_1}}$ and~$w_{j'_{r+i_2}}$, overlap.
Thus $j'_{r+i_1}=j'_{r+i_2}$.
Now suppose that $j^*$ is their common value.
Then
\begin{equation}\label{eq2.39}
Q_{r+i_1}R_{r+i_1}=w_{j^*}(0,y_{r+i_1})
\quad\mbox{and}\quad
Q_{r+i_2}R_{r+i_2}=w_{j^*}(0,y_{r+i_2})
\end{equation}
Furthermore, since there are precisely $s$ vertical edges on~$\PPP$, it follows that
\begin{equation}\label{eq2.40}
1\le i_1<i_2\le s+1.
\end{equation}
This means that we can take $i\le s+1$ in \eqref{eq2.37}, and explains our choice of the constant $c_1(A;s)$ given by \eqref{eq2.7}.

\begin{claim4}
Suppose that there exist integers $i_1$ and $i_2$ satisfying \eqref{eq2.40} such that the following conditions hold:

(1)
For every integer $i$ satisfying $1\le i\le i_2$, there exists an integer $k_{r+i}$ satisfying \eqref{eq2.38}
such that the top interval $Q_{r+i}R_{r+i}$ given by \eqref{eq2.34} arises when the $k_{r+i}$-th shift
under the reverse $\alpha$-flow of the open interval $Q_{r+i-1}R_{r+i-1}$ splits for the first time, where $Q_rR_r=Q^{(r)}R^{(r)}$.

(2)
For every integer $i$ satisfying $1\le i\le i_2$, the conditions \eqref{eq2.35} and \eqref{eq2.36} hold.

(3)
There exists an integer $j^*$ such that the condition \eqref{eq2.39} holds.

Then there is a visiting time $t^*$ such that $0<\vert t^*\vert\le\vert t_{r+i_2}-t_{r+i_1}\vert$ and $\LLL(t^*)\in QR$, where
$R_{r+i_1}=\LLL(t_{r+i_1})$ and $R_{r+i_2}=\LLL(t_{r+i_2})$, so the conclusion of Lemma~\ref{lem22}
holds with a suitable constant $c_0(A;s)$.
\end{claim4}

\begin{proof}[Justification of Claim~4]
Since $i_1<i_2$, we have $y_{r+i_1}>y_{r+i_2}$.
It follows from \eqref{eq2.39} that
\textcolor{white}{xxxxxxxxxxxxxxxxxxxxxxxxxxxxxx}
\begin{displaymath}
R_{r+i_2}=\LLL(t_{r+i_2})\in Q_{r+i_1}R_{r+i_1}.
\end{displaymath}
Applying the forward $\alpha$-flow for time $t_r-t_{r+i_1}$ then takes $Q_{r+i_1}R_{r+i_1}$ to~$Q^{(r)}R^{(r)}$,
and also $R_{r+i_2}=\LLL(t_{r+i_2})$ to $\LLL(t_r+t_{r+i_2}-t_{r+i_1})$, so that $\LLL(t_r+t_{r+i_2}-t_{r+i_1})\in Q^{(r)}R^{(r)}$.
Applying next the reverse $\alpha$-flow for time $t_r$ then takes $Q^{(r)}R^{(r)}$ to~$QR$,
and also $\LLL(t_r+t_{r+i_2}-t_{r+i_1})$ to $\LLL(t_{r+i_2}-t_{r+i_1})$, so that $\LLL(t_{r+i_2}-t_{r+i_1})\in QR$.
This justifies the first assertion in Claim~4.
Next, note that the open interval $Q_{r+i_2}R_{r+i_2}$ arises as a consequence of
\begin{align}
k_{r+i_1+1}+\ldots+k_{r+i_2}
&
\le\sum_{i=0}^s\frac{9s^2(A+2)}{y_{r+i}}
\le\sum_{i=0}^s\frac{9s^2(A+2)}{\delta_i(A;s)x_r}
\nonumber
\\
&
\le\frac{1}{x}\sum_{i=0}^s\frac{9s^2(A+2)(36s^2(A+2)^2)^i}{(c_1(A;s))^s}
\nonumber
\end{align}
consecutive shifts under the reverse $\alpha$-flow of the open interval~$Q_{r+i_1}R_{r+i_1}$, using \eqref{eq2.17}, \eqref{eq2.35} and \eqref{eq2.36}.
Since each shift under the reverse $\alpha$-flow corresponds to a geodesic segment of length $\sqrt{1+\alpha^2}\le\sqrt{2}$, it follows that
\begin{displaymath}
\vert t_{r+i_2}-t_{r+i_1}\vert\le\frac{1}{x}\sum_{i=0}^s\frac{9\sqrt{2}s^2(A+2)(36s^2(A+2)^2)^i}{(c_1(A;s))^s}.
\end{displaymath}
If the constant $c_0(A;s)$ in Lemma~\ref{lem22} is chosen to satisfy
\begin{equation}\label{eq2.41}
c_0(A;s)\ge\sum_{i=0}^s\frac{9\sqrt{2}s^2(A+2)(36s^2(A+2)^2)^i}{(c_1(A;s))^s},
\end{equation}
then the conclusion of Lemma~\ref{lem22} holds.
\end{proof}

Lemma~\ref{lem22} now follows if we choose $c_0(A;s)$ sufficiently large to satisfy \eqref{eq2.15} and \eqref{eq2.41}.
\end{proof}

We have the following simple corollary of Lemma~\ref{lem22}.

\begin{lem}\label{lem23}
Under the hypotheses of Lemma~\ref{lem22}, the distance between the point $\LLL(t_0)$ and either endpoint $Q$ or $R$ is at least
\begin{displaymath}
\frac{x}{(A+2)c_0(A;s)}.
\end{displaymath}
\end{lem}

\begin{proof}
Note that the vertical distance between $R$ and $\LLL(t_0)$ is of the form
\begin{displaymath}
\{n\alpha\}\ge\Vert n\alpha\Vert\ge\frac{1}{(A+2)n},
\end{displaymath}
where $n$ is the number of shifts under the $\alpha$-flow from $R$ to~$\LLL(t_0)$, using Lemma~\ref{lem21}.
On the other hand, it is clear that $n\le c_0(A;s)/x$.
A corresponding lower bound for the vertical distance between $Q$ and $\LLL(t_0)$ comes via a symmetry argument.
\end{proof}

For convenience, we write
\begin{displaymath}
c_2(A;s)=\frac{1}{(A+2)c_0(A;s)}.
\end{displaymath}
Note that this is a very small constant depending at most on $A$ and~$s$.

Let $w$ be the left vertical edge of a fixed atomic square of the polysquare surface~$\PPP$.
Consider a finite segment
\begin{equation}\label{eq2.42}
\Gamma(\sigma;T)=\{\LLL(t):0\le\vert\sigma-t\vert\le T\},
\end{equation}
centered at~$\sigma$, of the geodesic $\LLL(t)$ of slope~$\alpha$.
Suppose that this segment intersects $w$ at $N=N(T)$ points.
Denote these points by
\begin{displaymath}
w(y_i)=\LLL(t_i)\in w,
\quad
1\le i\le N=N(T),
\end{displaymath}
and arrange them in increasing order
\begin{equation}\label{eq2.43}
0<y_{i_1}<y_{i_2}<\ldots<y_{i_N}<1.
\end{equation}
We then define the maximum gap of \eqref{eq2.43} by
\begin{equation}\label{eq2.44}
\mg(\Gamma(\sigma;T;w))=\max_{0\le n\le N}(y_{i_{n+1}}-y_{i_n}),
\end{equation}
with the convention that $y_{i_0}=0$ and $y_{i_{N+1}}=1$.
Using this concept, we can establish an extension to Lemma~\ref{lem23} as follows.

\begin{lem}\label{lem24}
For any finite segment $\Gamma(\sigma;T)$, given by \eqref{eq2.42}, of a geodesic $\LLL(t)$ of badly approximable slope~$\alpha$, let
$\mg(\Gamma(\sigma;T;w))=x$.
Then the longer finite segment
\textcolor{white}{xxxxxxxxxxxxxxxxxxxxxxxxxxxxxx}
\begin{displaymath}
\Gamma\left(\sigma;T+\frac{c_0(A;s)}{x}\right)=\left\{\LLL(t):0\le\vert\sigma-t\vert\le T+\frac{c_0(A;s)}{x}\right\}
\end{displaymath}
has the property that
\begin{displaymath}
\mg\left(\Gamma\left(\sigma;T+\frac{c_0(A;s)}{x};w\right)\right)\le(1-c_2(A;s))x.
\end{displaymath}
\end{lem}

Iterating Lemma~\ref{lem24} sufficiently many times, we obtain the following.

\begin{lem}\label{lem25}
Under the hypotheses of Lemma~\ref{lem24}, there exists a positive constant $c_3(A;s)$ such that the longer finite segment
\begin{displaymath}
\Gamma\left(\sigma;T+\frac{c_3(A;s)}{x}\right)=\left\{\LLL(t):0\le\vert\sigma-t\vert\le T+\frac{c_3(A;s)}{x}\right\}
\end{displaymath}
has the property that
\begin{displaymath}
\mg\left(\Gamma\left(\sigma;T+\frac{c_3(A;s)}{x};w\right)\right)\le\frac{x}{2}.
\end{displaymath}
\end{lem}

\begin{proof}[Proof of Theorem~\ref{thm1}]
Superdensity of a geodesic with badly approximable slope on a polysquare surface $\PPP$
is a straightforward deduction from the \textit{discrete} superdensity of intersection points on any fixed vertical edge $w$ of~$\PPP$,
so it remains to establish the latter.

Suppose that a half-infinite geodesic $\LLL(t)$ with badly approximable slope $\alpha$ visits a point $R$ of~$w$, where
\textcolor{white}{xxxxxxxxxxxxxxxxxxxxxxxxxxxxxx}
\begin{displaymath}
R=w(y)=\LLL\left(\frac{M}{2}\right),
\end{displaymath}
where $M>0$ is sufficiently large.
Note that this allows us to move forward and backward in time from $R$ by up to $M/2$ and still stay within the interval $(0,M)$.

Consider the finite segment $\Gamma(\sigma;T)$, given by \eqref{eq2.42}, with $\sigma=M/2$ and $T=0$,
so that $\Gamma(\sigma;T)$ contains precisely one point~$R$.
Then $\mg(\Gamma(\sigma;T;w))\le1$.
On applying Lemma~\ref{lem25} with $x=1$, we deduce that
\begin{displaymath}
\mg\left(\Gamma\left(\frac{M}{2};\frac{c_3(A;s)}{1};w\right)\right)\le\frac{1}{2}.
\end{displaymath}
Using this and applying Lemma~\ref{lem25} again with $x=1/2$, we deduce that
\begin{displaymath}
\mg\left(\Gamma\left(\frac{M}{2};c_3(A;s)+\frac{c_3(A;s)}{1/2};w\right)\right)\le\frac{1}{4}.
\end{displaymath}
Using this and applying Lemma~\ref{lem25} again with $x=1/4$, we deduce that
\begin{displaymath}
\mg\left(\Gamma\left(\frac{M}{2};3c_3(A;s)+\frac{c_3(A;s)}{1/4};w\right)\right)\le\frac{1}{8}.
\end{displaymath}
And so on.
In general, we have
\begin{displaymath}
\mg\left(\Gamma\left(\frac{M}{2};(2^n-1)c_3(A;s);w\right)\right)\le\frac{1}{2^n}
\end{displaymath}
for every integer $n\ge1$ such that $(2^n-1)c_3(A;s)\le M/2$.
This clearly proves superdensity of the intersection points on any vertical edge $w$ of~$\PPP$, and completes the proof of Theorem~\ref{thm1}.
\end{proof}

%
%

\section{Adaptation to the regular octagon surface}\label{sec3}

As mentioned in Section~\ref{sec1}, $1$-direction geodesic flow on a finite polysquare surface \textit{modulo one}
becomes a torus line flow on $[0,1)^2$.
We can view this observation as a lucky reduction, since torus line flow on $[0,1)^2$ gives rise to an integrable system
with a basically complete theory.

We now study flat systems that do not enjoy such lucky reduction,
and begin with arguably the simplest non-integrable billiard in the $\pi/8$ right triangle.
Applying the well known technique of unfolding, this can be shown to be equivalent
to the problem of linear flow on the regular octagon surface; see Figure~3.1.

\begin{displaymath}
\begin{array}{c}
\includegraphics{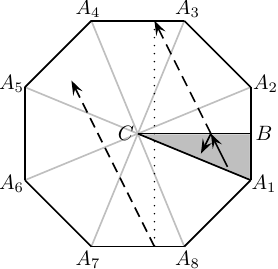}
\\
\mbox{Figure 3.1: billiard in the $\pi/8$ right triangle and the equivalent problem}
\\
\mbox{of linear flow on the regular octagon surface}
\end{array}
\end{displaymath}

To turn the regular octagon region into the regular octagon surface, we identity opposite parallel edges, so that there are $4$ pairs
\begin{displaymath}
(A_1A_2,A_6A_5),
\quad
(A_2A_3,A_7A_6),
\quad
(A_3A_4,A_8A_7),
\quad
(A_4A_5,A_1A_8)
\end{displaymath}
of identified edges.

We obtain a $1$-direction geodesic flow on this compact orientable surface, which is a $16$-fold covering
of a billiard orbit in the $\pi/8$ right triangle region.
The unfolding process goes as follows.
Let $C$ denote the centre of the octagon.
The right triangle $A_1BC$ has angle $\pi/8$ at the vertex~$C$.
Reflecting the triangle $A_1BC$ across the side $BC$ is the first step of the unfolding process,
and gives rise to the image $A_2BC$ which together with the original triangle $A_1BC$ forms the triangle $A_1A_2C$.
We now reflect the triangle $A_1A_2C$ across the side $A_2C$ to obtain the image $A_3A_2C$,
then reflect the triangle $A_2A_3C$ across the side $A_3C$ to obtain the image $A_4A_3C$, and so on,
until we end up with the regular octagon.
Non-integrability is clear from the unfolding,
since the vertices of the octagon are split-singularities of the $1$-direction geodesic flow on the surface.

The same argument can be applied to billiard in the $\pi/n$ right triangle for every even integer $n\ge4$,
and we can show that this is equivalent to the problem of linear flow on the regular $n$-gon surface,
obtained from the regular $n$-gon region by identifying opposite parallel edges, so that there are $n/2$ pairs of identified edges.
These are non-integrable systems for every even integer $n\ge8$.
The cases $n=4$ and $n=6$ give rise to integrable systems.

\begin{remark}
Consider the regular $n$-gon surface with even integer $n\ge4$.
If $n$ is divisible by~$4$, then boundary identification gives rise to $1$ vertex, $n/2$ edges and $1$ region, so it follows from Euler's formula
\begin{displaymath}
2-2g=V-E+R=1-\frac{n}{2}+1
\end{displaymath}
that the genus $g=n/4$.
If $n$ is not divisible by~$4$, then boundary identification gives rise to $2$ vertices, $n/2$ edges and $1$ region, so it follows from Euler's formula
\begin{displaymath}
2-2g=V-E+R=2-\frac{n}{2}+1
\end{displaymath}
that the genus $g=(n-2)/4$.
Thus the genus $g=1$ when $n=4$ or $n=6$, for each of which the geodesic flow is integrable,
consistent with the well known fact that we can tile the plane with squares or regular hexagons.
On the other hand, the genus $g>1$ when $n\ge8$, consistent with the well known fact that
we cannot tile the plane with regular $n$-gons when the even integer $n\ge8$.
\end{remark}

Let us return to the regular octagon surface.
While it looks completely different from a polysquare surface, there is a hidden similarity.
The regular octagon surface is in fact equivalent to a \textit{polyrectangle surface}; see Figure~3.2.

\begin{displaymath}
\begin{array}{c}
\includegraphics{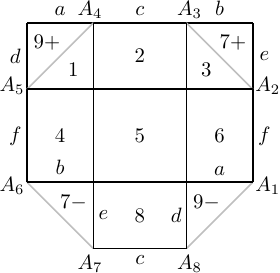}
\\
\mbox{Figure 3.2: the regular octagon surface viewed as a polyrectangle surface}
\end{array}
\end{displaymath}

The edge $A_2A_3$ is identified with the edge~$A_7A_6$.
This allows us to replace the triangle labelled $7-$ by the triangle labelled $7+$, with the two horizontal edges $b$
identified and the two vertical edges $e$ identified.
Likewise, the edge $A_4A_5$ is identified with the edge~$A_1A_8$.
This allows us to replace the triangle labelled $9-$ by the triangle labelled $9+$, with the two horizontal edges $a$
identified and the two vertical edges $d$ identified.
Thus the regular octagon surface becomes a polyrectangle surface consisting of $7$ rectangles, labelled $(1,9+),2,(3,7+),4,5,6,8$.
With the edge identification, this polyrectangle surface has $2$ horizontal streets
\begin{equation}\label{eq3.1}
(1,9+),2,(3,7+),8
\quad\mbox{and}\quad
4,5,6.
\end{equation}
Furthermore, if we assume that the first horizontal street in \eqref{eq3.1} has rectangles with vertical edges of length~$1$,
then the second horizontal street in \eqref{eq3.1} has rectangles with vertical edges of length~$\sqrt{2}$.

Let us relabel the vertical edges of the polyrectangle surface by $w_i$, $i=1,\ldots,7$, as shown in Figure~3.3.

As in Section~\ref{sec2}, we let $w_i=w_i[0,1]$, $i=1,2,3,4$,
denote the parametrizations of the $4$ vertical edges of the first horizontal street,
with $0$ denoting the bottom endpoint and $1$ denoting the top endpoint, and $w_i=w_i[0,\sqrt{2}]$, $i=5,6,7$,
denote the parametrizations of the $3$ vertical edges of the second horizontal street,
with $0$ denoting the bottom endpoint and $\sqrt{2}$ denoting the top endpoint.

\begin{displaymath}
\begin{array}{c}
\includegraphics{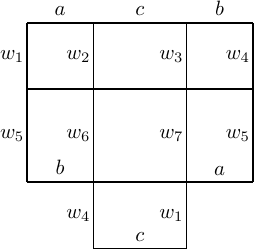}
\\
\mbox{Figure 3.3: the vertical edges of the regular octagon surface}
\\
\mbox{represented as a polyrectangle surface}
\end{array}
\end{displaymath}

Like before, we consider a subinterval $S$ on a vertical edge of the polyrectangle surface.
The $\alpha$-flow shifts $S$ until it hits some other vertical edge or edges of the surface for the first time, with image~$S(\alpha)$, say.
Likewise, we can look at the effect of such a shift under the $\alpha$-flow on a point $w_i(y)\in w_i$, where $i=1,\ldots,7$.

In view of symmetry, we may assume, without loss of generality, that $0<\alpha<1$.

It is clear that the upward vertical travel under such a shift is either $\alpha$ or~$\sqrt{2}\alpha$, depending on the edge where the shift begins.
Thus the total upward vertical travel after $n^*$ successive such shifts is given by
\begin{displaymath}
n_1\alpha+n_2\sqrt{2}\alpha,
\end{displaymath}
where $n_1$ denotes the number of shifts from the edges $w_1,w_3,w_5,w_7$ and $n_2$
denotes the number of shifts from the edges $w_2,w_4,w_6$, so that
\begin{displaymath}
n_1,n_2\ge0
\quad\mbox{and}\quad
n_1+n_2=n^*.
\end{displaymath}
Let $n_3$ denote the total number of times when a shift moves a point on a long vertical edge $w_5,w_6,w_7$
to a point on a short vertical edge $w_1,w_2,w_3,w_4$, so that
\begin{displaymath}
0\le n_3\le n^*.
\end{displaymath}
Suppose now that these $n^*$ successive shifts take us from some point $w_{i'}(y)\in w_{i'}$ to some point $w_{i''}(y^*)\in w_{i''}$.
Then there exists some integer $n_4$ such that
\begin{displaymath}
y+n_1\alpha+n_2\sqrt{2}\alpha\in
\left\{\begin{array}{ll}
{[n_3\sqrt{2}+n_4,n_3\sqrt{2}+n_4+1)},
&\mbox{if $i''=1,2,3,4$},\\
{[n_3\sqrt{2}+n_4,n_3\sqrt{2}+n_4+\sqrt{2})},
&\mbox{if $i''=5,6,7$},
\end{array}\right.
\end{displaymath}
It is then absolutely clear that
\begin{displaymath}
y+n_1\alpha+n_2\sqrt{2}\alpha=y^*+n_3\sqrt{2}+n_4.
\end{displaymath}
This leads to the crucial expression
\begin{equation}\label{eq3.2}
y^*-y=n_1\alpha+n_2\sqrt{2}\alpha-n_3\sqrt{2}-n_4,
\end{equation}
and this allows us to make use of the quantity
\begin{displaymath}
\Vert n_1\alpha+n_2\sqrt{2}\alpha-n_3\sqrt{2}\Vert,
\end{displaymath}
where $0\le n_1,n_2,n_3\le n^*$ and $n_1+n_2\ge1$.

For the remainder of this section, we assume that the slope $\alpha$ satisfies
\begin{displaymath}
0<\alpha=\frac{\root{3}\of{3}}{2}<1,
\end{displaymath}
and consider the algebraic number field
\begin{displaymath}
K=\Qq(\sqrt{2},\root{3}\of{3}),
\end{displaymath}
of degree~$6$, which is the extension of $\Qq$ by $\sqrt{2}$ and $\root{3}\of{3}=2\alpha$.
Note that $K$ contains the elements $\alpha$, $\sqrt{2}$ and~$\sqrt{2}\alpha$.

Suppose that for some integer~$n_0$,
\begin{displaymath}
\Vert n_1\alpha+n_2\sqrt{2}\alpha-n_3\sqrt{2}\Vert=\vert n_1\alpha+n_2\sqrt{2}\alpha-n_3\sqrt{2}-n_0\vert=\omega,
\end{displaymath}
so that $0<\omega\le1/2$.
Then there exists a constant $c_4=c_4(\alpha)>0$ such that $\gamma=c_4\omega$ is an algebraic integer in~$K$.
Let $\gamma_1,\ldots,\gamma_6$, denote the conjugates of $\gamma$ in~$K$, with $\gamma_1=\gamma$.
Then the norm of~$\gamma$, given by $N(\gamma)=\gamma_1\ldots\gamma_6$, is a non-zero integer in~$\Zz$.
Thus $\vert N(\gamma)\vert\ge1$, and this implies that
\begin{displaymath}
\omega=\frac{\gamma}{c_4}\ge\frac{1}{c_4\vert\gamma_2\ldots\gamma_6\vert}.
\end{displaymath}
This leads to the following analog of Lemma~\ref{lem21}.

\begin{lem}\label{lem31}
Let $\alpha=\root{3}\of{3}/2$.
Then there exists a constant $c_5=c_5(\alpha)>0$ such that for any integers $n_1,n_2,n_3$ with $n_1^2+n_2^2\ge1$, we have
\begin{displaymath}
\Vert n_1\alpha+n_2\sqrt{2}\alpha-n_3\sqrt{2}\Vert>\frac{c_5}{N^5},
\end{displaymath}
where $N=\max\{\vert n_1\vert,\vert n_2\vert,\vert n_3\vert\}$.
\end{lem}

We develop here a rather straightforward adaptation of the method in Section~\ref{sec2}.
Since Lemma~\ref{lem31} gives a much weaker bound than Lemma~\ref{lem21}, we cannot expect to be able to establish superdensity here.
Nevertheless, we can still establish \textit{polynomial} time-quantitative density.

\begin{thm}\label{thm2}
Let $\alpha=\root{3}\of{3}/2$.
Consider a half-infinite geodesic $\LLL(t)$, $t\ge0$, of slope $\alpha$ and with the usual arc-length parametrization,
on the regular octagon surface as represented by the polyrectangle surface $\PPP$ as shown in Figure~3.3.
Then there are explicitly computable constants $c_6=c_6(\alpha)>1$ and $c_7=c_7(\alpha)>1$ such that,
for any integer $n\ge2$ and any aligned square $A$ of side length $1/n$ on~$\PPP$, there exists a real number $t_0$ such that
\begin{displaymath}
0\le t_0\le c_6n^{c_7}
\quad\mbox{and}\quad
\LLL(t_0)\in A.
\end{displaymath}
\end{thm}

Consider the geodesic $\LLL(t)=\LLL_\alpha(t)$ with slope $\alpha$ and starting point $\LLL(0)=R$,
where $R$ lies on a vertical edge $w_{i_0}$ of the polyrectangle surface~$\PPP$.
Assume that $\LLL(t)$ has arc-length parametrization, and that it does not hit a vertex of $\PPP$
over a sufficiently long neighborhood $-T\le t\le T$ of~$0$.
Then $R=w_{i_0}(y)$ for some $y$ satisfying $0<y<1$ if $i_0=1,2,3,4$ and satisfying $0<y<\sqrt{2}$ if $i_0=5,6,7$.
Let $Q=w_{i_0}(z)$ be a point where $0<z<y$.
We study the following question.
What can we say about the time $t$ with $\LLL(t)\in QR$ such that $\vert t\vert$ is minimum?
Here $QR$ denotes the open interval with endpoints $Q$ and~$R$.
In other words, how long does it take for the geodesic, starting at the point~$R$, to visit the open interval $QR$ of length $x=y-z$,
if the geodesic can go both forward and backward?

Corresponding to Lemma~\ref{lem22}, we have the following intermediate result.

\begin{lem}\label{lem32}
Let $\alpha=\root{3}\of{3}/2$.
Let $QR$ be an open vertical segment, with top endpoint $R$ and length $0<x<1/2$,
on a vertical edge $w_{i_0}$ of the polyrectangle surface $\PPP$ as shown in Figure~3.3.
Consider a geodesic $\LLL(t)$ with slope $\alpha$ and starting point $\LLL(0)=R$.
There exists explicit constants $c_8=c_8(\alpha)>1$ and $c_9=c_9(\alpha)>1$
such that there is a $2$-direction visiting time $t^*$ satisfying
\begin{displaymath}
0<\vert t^*\vert\le\frac{c_8}{x^{c_9}}
\quad\mbox{and}\quad
\LLL(t^*)\in QR.
\end{displaymath}
\end{lem}

\begin{proof}
Let $Q=w_{i_0}(z)$ and $R=w_{i_0}(y)$, where $0<z<y<1$ if $i_0=1,2,3,4$, and where $0<z<y<\sqrt{2}$ if $i_0=5,6,7$.

If the shift of the open interval $QR$ under the $\alpha$-flow does not split, then there exists an integer $i_1$
satisfying $1\le i_1\le7$ such that $QR$ is shifted to an open interval $Q_1R_1$ on the vertical edge~$w_{i_1}$.
Let us now repeat the argument with the open interval~$Q_1R_1$.
If the shift of $Q_1R_1$ under the $\alpha$-flow does not split, then there exists an integer $i_2$ satisfying $1\le i_2\le7$
such that $Q_1R_1$ is shifted to an open interval $Q_2R_2$ on the vertical edge~$w_{i_2}$.
We now repeat the argument with the open interval~$Q_2R_2$, and so on, until we get the first split.

\begin{claim1}
Suppose that there is no split among the first $[(30)^6/c_5x^5]$ consecutive shifts of the open interval $QR$ under the $\alpha$-flow,
where $QR$ has length $0<x<1/2$, and $c_5=c_5(\alpha)$ is the constant in Lemma~\ref{lem31}.
Then there is a visiting time $t^*$ such that
\textcolor{white}{xxxxxxxxxxxxxxxxxxxxxxxxxxxxxx}
\begin{displaymath}
0<t^*\le\frac{2(30)^6}{c_5x^5}
\quad\mbox{and}\quad
\LLL(t^*)\in QR,
\end{displaymath}
so that the conclusion of Lemma~\ref{lem32} holds with suitable constants $c_8=c_8(\alpha)$ and $c_9=c_9(\alpha)$.
\end{claim1}

\begin{proof}[Justification of Claim~1]
The open intervals $Q_jR_j$, where $1\le j\le(30)^6/c_5x^5$, all have length~$x$.
Their total length is therefore at least
\begin{equation}\label{eq3.3}
\frac{(30)^6}{c_5x^4}-1\ge(4+3\sqrt{2})L,
\quad\mbox{where}\quad
L=\left[\frac{(30)^6}{9c_5x^4}\right].
\end{equation}
Note that the total length of the vertical edges of $\PPP$ is $4+3\sqrt{2}<9$.
It follows easily from the Pigeonhole Principle that there exists a point $P$ on a vertical edge
which is covered by at least $L$ of these open intervals~$Q_jR_j$.
In other words, there exist integers $j_\nu$, $\nu=1,\ldots,L$, such that
\begin{displaymath}
1\le j_1<j_2<\ldots<j_L\le\frac{(30)^6}{c_5x^5}
\end{displaymath}
such that $P$ is contained in
\textcolor{white}{xxxxxxxxxxxxxxxxxxxxxxxxxxxxxx}
\begin{displaymath}
Q_{j_\nu}R_{j_\nu},
\quad
\nu=1,\ldots,L.
\end{displaymath}
For every $\nu=1,\ldots,L$, the open interval $Q_{j_\nu}R_{j_\nu}$ lies on the vertical edge~$w_{i_{j_\nu}}$.
It follows that the values $i_{j_\nu}$, $\nu=1,\ldots,L$, are all equal to each other.
Suppose that $i^*$ is their common value.
Then for every $\nu=1,\ldots,L$, we can write
\begin{displaymath}
R_{j_\nu}=w_{i^*}(u_\nu),
\end{displaymath}
where $0<u_\nu<1$ or $0<u_\nu<\sqrt{2}$, and define $t_\nu$ by writing $\LLL(t_\nu)=R_{j_\nu}$.

Suppose first that there exist two integers $\nu'$ and $\nu''$ such that
\begin{equation}\label{eq3.4}
1\le\nu'<\nu''\le L
\quad\mbox{and}\quad
u_{\nu'}>u_{\nu''}.
\end{equation}
Since $Q_{j_{\nu'}}R_{j_{\nu'}}$ and $Q_{j_{\nu''}}R_{j_{\nu''}}$ intersect, we clearly have
\begin{displaymath}
R_{j_{\nu''}}=\LLL(t_{\nu''})\in Q_{j_{\nu'}}R_{j_{\nu'}}.
\end{displaymath}
Applying the reverse $\alpha$-flow for time~$t_{\nu'}$ then takes $Q_{j_{\nu'}}R_{j_{\nu'}}$ to~$QR$,
and also takes $R_{j_{\nu''}}=\LLL(t_{\nu''})$ to
$\LLL(t_{\nu''}-t_{\nu'})$, so that $\LLL(t_{\nu''}-t_{\nu'})\in QR$.
Now take $t^*=t_{\nu''}-t_{\nu'}>0$.
Then
\textcolor{white}{xxxxxxxxxxxxxxxxxxxxxxxxxxxxxx}
\begin{displaymath}
0<t^*\le\sqrt{2}\sqrt{1+\alpha^2}(j_{\nu''}-j_{\nu'})\le\frac{2(30)^6}{c_5x^5},
\end{displaymath}
justifying the claim.

Suppose next that there do not exist two integers $\nu'$ and $\nu''$ such that \eqref{eq3.4} holds.
Then we must have
\textcolor{white}{xxxxxxxxxxxxxxxxxxxxxxxxxxxxxx}
\begin{displaymath}
u_1<u_2<\ldots<u_L
\quad\mbox{and}\quad
u_L-u_1\le x,
\end{displaymath}
and a routine average computation argument shows that for at least $2L/3$ of the indices $\nu=1,\ldots,L$, we have
\textcolor{white}{xxxxxxxxxxxxxxxxxxxxxxxxxxxxxx}
\begin{equation}\label{eq3.5}
u_{\nu+1}-u_\nu\le\frac{3x}{L}.
\end{equation}
On the other hand, we also have
\begin{displaymath}
j_1<j_2<\ldots<j_L
\quad\mbox{and}\quad
j_L-j_1\le\frac{(30)^6}{c_5x^5},
\end{displaymath}
and a routine average computation argument shows that for at least $2L/3$ of the indices $\nu=1,\ldots,L$, we have
\textcolor{white}{xxxxxxxxxxxxxxxxxxxxxxxxxxxxxx}
\begin{equation}\label{eq3.6}
j_{\nu+1}-j_\nu\le\frac{30}{x}.
\end{equation}
It follows that there must exist some index $\nu=1,\ldots,L$ such that both \eqref{eq3.5} and \eqref{eq3.6} hold.
For this value of~$\nu$, it follows from \eqref{eq3.2} that there exist integers $n_1,n_2,n_3,n_4$ such that
\begin{equation}\label{eq3.7}
u_{\nu+1}-u_\nu=n_1\alpha+n_2\sqrt{2}\alpha-n_3\sqrt{2}-n_4
\end{equation}
and $0\le n_1,n_2,n_3\le j_{\nu+1}-j_\nu$ and $n_1+n_2\ge1$.
Using Lemma~\ref{lem31} and \eqref{eq3.5}--\eqref{eq3.7}, we deduce that
\begin{displaymath}
\frac{3x}{L}
\ge\Vert u_{\nu+1}-u_\nu\Vert
=\Vert n_1\alpha+n_2\sqrt{2}\alpha-n_3\sqrt{2}\Vert
>\frac{c_5}{(j_{\nu+1}-j_\nu)^5}
\ge\frac{c_5x^5}{(30)^5}.
\end{displaymath}
However, this leads to the inequality
\begin{displaymath}
L<\frac{(30)^6}{10c_5x^4}
\end{displaymath}
which clearly contradicts the definition of $L$ given by \eqref{eq3.3}.

It follows that there exist two integers $\nu'$ and $\nu''$ such that \eqref{eq3.4} holds, and this completes our justification of Claim~1.
\end{proof}

In view of Claim~1, we may assume that there exists an integer $k$ such that
\begin{displaymath}
1\le k\le\frac{(30)^6}{c_5x^5}
\end{displaymath}
and the $k$-th shift under the $\alpha$-flow of the open interval $QR$ of length $x$ splits for the first time.

Suppose that the image of the original open interval $QR$ after the first $k$ shifts under the $\alpha$-flow
now consists of a vertex of~$\PPP$, a top interval $w_{j_1}(0,x_1)$ of length~$x_1$ and a bottom interval of length~$x_1^*$,
where $x_1+x_1^*=x$.
Since the starting point $R$ of the geodesic is the \textit{top endpoint} of the interval $QR=w_{i_0}(z,y)$,
we shall make use of top intervals in our subsequent argument.
We distinguish two cases.
Either
\begin{equation}\label{eq3.8}
x_1\ge c_{10}x^{5^{16}}
\quad\mbox{or}\quad
0<x_1<c_{10}x^{5^{16}},
\end{equation}
where the choice of the constant
\begin{equation}\label{eq3.9}
c_{10}=c_{10}(\alpha)=\left(\frac{c_5(\alpha)}{60}\right)^{6^{16}}
\end{equation}
and the choice of the exponent $5^{16}$ for $x$ will be explained later.

Suppose that the first case in \eqref{eq3.8} holds.
Then we delete the bottom interval, keep the top interval $w_{j_1}(0,x_1)$ and write $Q^{(1)}R^{(1)}=w_{j_1}(0,x_1)$.
It then follows from our construction that the geodesic $\LLL(t)$, $t\ge0$, starting at the point~$R$, contains the point~$R^{(1)}$, so that
\textcolor{white}{xxxxxxxxxxxxxxxxxxxxxxxxxxxxxx}
\begin{displaymath}
R^{(1)}=\LLL(t_1)
\quad\mbox{for some $t_1>0$}.
\end{displaymath}

We now repeat this argument on the open interval $Q^{(1)}R^{(1)}=w_{j_1}(0,x_1)$.

In view of a suitable analog of Claim~1, we may assume that there exists an integer $k_1$ such that
\textcolor{white}{xxxxxxxxxxxxxxxxxxxxxxxxxxxxxx}
\begin{displaymath}
1\le k_1\le\frac{(30)^6}{c_5x_1^5}
\end{displaymath}
and the $k_1$-th shift under the $\alpha$-flow of the open interval $Q^{(1)}R^{(1)}$ of length $x_1$ splits for the first time.

Suppose that the image of the original open interval $Q^{(1)}R^{(1)}$ after the first $k_1$ shifts under the $\alpha$-flow
now consists of a vertex of~$\PPP$, a top interval $w_{j_2}(0,x_2)$ of length~$x_2$ and a bottom interval of length~$x_2^*$,
where $x_2+x_2^*=x_1$.
We distinguish two cases.
Either
\textcolor{white}{xxxxxxxxxxxxxxxxxxxxxxxxxxxxxx}
\begin{equation}\label{eq3.10}
x_2\ge c_{10}x_1^{5^{16}}
\quad\mbox{or}\quad
0<x_2<c_{10}x_1^{5^{16}}.
\end{equation}

Suppose that the first case in \eqref{eq3.10} holds.
Then we delete the bottom interval, keep the top interval $w_{j_2}(0,x_2)$ and write $Q^{(2)}R^{(2)}=w_{j_2}(0,x_2)$.
It then follows from our construction that the geodesic $\LLL(t)$, $t\ge0$, starting at the point~$R$, contains the point~$R^{(2)}$, so that
\textcolor{white}{xxxxxxxxxxxxxxxxxxxxxxxxxxxxxx}
\begin{displaymath}
R^{(2)}=\LLL(t_2)
\quad\mbox{for some $t_2>t_1$}.
\end{displaymath}

We now repeat this argument on the open interval $Q^{(2)}R^{(2)}=w_{j_2}(0,x_2)$.

In view of a suitable analog of Claim~1, we may assume that there exists an integer $k_2$ such that
\textcolor{white}{xxxxxxxxxxxxxxxxxxxxxxxxxxxxxx}
\begin{displaymath}
1\le k_2\le\frac{(30)^6}{c_5x_2^5}
\end{displaymath}
and the $k_2$-th shift under the $\alpha$-flow of the open interval $Q^{(2)}R^{(2)}$ of length $x_2$ splits for the first time.

Suppose that the image of the original open interval $Q^{(2)}R^{(2)}$ after the first $k_2$ shifts under the $\alpha$-flow
now consists of a vertex of~$\PPP$, a top interval $w_{j_3}(0,x_3)$ of length~$x_3$ and a bottom interval of length~$x_3^*$,
where $x_3+x_3^*=x_2$.
We distinguish two cases.
Either
\textcolor{white}{xxxxxxxxxxxxxxxxxxxxxxxxxxxxxx}
\begin{equation}\label{eq3.11}
x_3\ge c_{10}x_2^{5^{16}}
\quad\mbox{or}\quad
0<x_3<c_{10}x_2^{5^{16}}.
\end{equation}

Suppose that the first case in \eqref{eq3.11} holds.
Then we delete the bottom interval, keep the top interval $w_{j_3}(0,x_3)$ and write $Q^{(3)}R^{(3)}=w_{j_3}(0,x_3)$.
It then follows from our construction that the geodesic $\LLL(t)$, $t\ge0$, starting at the point~$R$, contains the point~$R^{(3)}$, so that
\textcolor{white}{xxxxxxxxxxxxxxxxxxxxxxxxxxxxxx}
\begin{displaymath}
R^{(3)}=\LLL(t_3)
\quad\mbox{for some $t_3>t_2$}.
\end{displaymath}

We now repeat this argument on the open interval $Q^{(3)}R^{(3)}=w_{j_3}(0,x_3)$.

And so on, \textit{assuming that at each step, the first case in the corresponding analog
of \eqref{eq3.8}, \eqref{eq3.10} and \eqref{eq3.11} holds}.

This forward shift process under the $\alpha$-flow defines a sequence of top intervals
\begin{equation}\label{eq3.12}
Q^{(i)}R^{(i)}=w_{j_i}(0,x_i),
\quad
i\ge1,
\end{equation}
each of which arises when the $k_{i-1}$-th shift under the $\alpha$-flow of the open interval $Q^{(i-1)}R^{(i-1)}$ of length $x_{i-1}$
splits for the first time, and the integer $k_{i-1}$ satisfies
\begin{equation}\label{eq3.13}
1\le k_{i-1}\le\frac{(30)^6}{c_5x_{i-1}^5}.
\end{equation}
The lengths $x_i$ of these intervals \eqref{eq3.12} satisfy
\begin{equation}\label{eq3.14}
x_i\ge c_{10}x_{i-1}^{5^{16}},
\end{equation}
with the convention that $x_0=x$.
Furthermore, the geodesic $\LLL(t)$, $t\ge0$, starting at the point~$R$, contains the point~$R^{(i)}$, so that
\begin{displaymath}
R^{(i)}=\LLL(t_i)
\quad\mbox{for some $t_i>t_{i-1}$},
\end{displaymath}
where $t_0=0$.

As there are only finitely many vertical edges in the polyrectangle surface~$\PPP$, there will at some point be \textit{edge repetition},
when there exist two integers $i_1$ and $i_2$ satisfying $1\le i_1<i_2$ such that the corresponding top intervals
\begin{displaymath}
Q^{(i_1)}R^{(i_1)}=w_{j_{i_1}}(0,x_{i_1})
\quad\mbox{and}\quad
Q^{(i_2)}R^{(i_2)}=w_{j_{i_2}}(0,x_{i_2})
\end{displaymath}
lying respectively on the vertical edges $w_{j_{i_1}}$ and~$w_{j_{i_2}}$, overlap.
Thus $j_{i_1}=j_{i_2}$.
Now suppose that $j^*$ is their common value.
Then
\begin{equation}\label{eq3.15}
Q^{(i_1)}R^{(i_1)}=w_{j^*}(0,x_{i_1})
\quad\mbox{and}\quad
Q^{(i_2)}R^{(i_2)}=w_{j^*}(0,x_{i_2}).
\end{equation}
Furthermore, since there are precisely $7$ vertical edges on~$\PPP$, it follows that
\begin{equation}\label{eq3.16}
1\le i_1<i_2\le8.
\end{equation}

\begin{claim2}
Suppose that there exist integers $i_1$ and $i_2$ satisfying \eqref{eq3.16} such that the following conditions hold:

(1)
For every integer $i$ satisfying $1\le i\le i_2$, there exists an integer $k_{i-1}$ satisfying \eqref{eq3.13}
such that the top interval $Q^{(i)}R^{(i)}$ given by \eqref{eq3.12} arises when the $k_{i-1}$-th shift
under the $\alpha$-flow of the open interval $Q^{(i-1)}R^{(i-1)}$ splits for the first time, where $Q^{(0)}R^{(0)}=QR$.

(2)
For every integer $i$ satisfying $1\le i\le i_2$, the condition \eqref{eq3.14} holds, where $x_0=x$.

(3)
There exists an integer $j^*$ such that the condition \eqref{eq3.15} holds.

Then there is a visiting time $t^*$ such that $0<t^*\le t_{i_2}$ and $\LLL(t^*)\in QR$, where $R^{(i_2)}=\LLL(t_{i_2})$,
and the conclusion of Lemma~\ref{lem32} holds with suitable constants $c_8=c_8(\alpha)$ and $c_9=c_9(\alpha)$.
\end{claim2}

\begin{proof}[Justification of Claim~2]
Since $i_1<i_2$, we have $x_{i_1}>x_{i_2}$.
It follows from \eqref{eq3.15} that
\begin{displaymath}
R^{(i_2)}=\LLL(t_{i_2})\in Q^{(i_1)}R^{(i_1)}.
\end{displaymath}
Applying the reverse $\alpha$-flow for time $t_{i_1}$ then takes $Q^{(i_1)}R^{(i_1)}$ to~$QR$,
and also takes $R^{(i_2)}=\LLL(t_{i_2})$ to $\LLL(t_{i_2}-t_{i_1})$, so that $\LLL(t_{i_2}-t_{i_1})\in QR$.
This justifies the first assertion in Claim~2.
Next, note that the open interval $Q^{(i_2)}R^{(i_2)}$ arises as a consequence of
\begin{displaymath}
k+k_1+\ldots+k_{i_2-1}
\le\frac{(30)^6}{c_5}\sum_{i=0}^7x_i^{-5}
\end{displaymath}
consecutive shifts under the $\alpha$-flow of the open interval~$QR$, using \eqref{eq3.13}.
Since each shift under the $\alpha$-flow corresponds to a geodesic segment of length $\sqrt{1+\alpha^2}$ or $\sqrt{2}\sqrt{1+\alpha^2}$,
both less than~$2$, it follows that
\begin{equation}\label{eq3.17}
t_{i_2}\le\frac{2(30)^6}{c_5}\sum_{i=0}^7x_i^{-5}.
\end{equation}
Finally, note that the finite sum in \eqref{eq3.17} can be bounded by a polynomial in $x^{-1}$ with non-negative coefficients
depending at most on $c_{10}=c_{10}(\alpha)$, in view of \eqref{eq3.14}.
It is then clear that the conclusion of Lemma~\ref{lem32} holds with suitably chosen constants $c_8=c_8(\alpha)$ and $c_9=c_9(\alpha)$.
\end{proof}

Suppose next that before edge repetition takes place, the condition \eqref{eq3.14} fails.
More precisely, suppose that $r$ satisfying $0\le r\le7$ is the smallest integer $i$ such that $x_{i+1}<c_{10}x_i^{5^{16}}$.
Then
\textcolor{white}{xxxxxxxxxxxxxxxxxxxxxxxxxxxxxx}
\begin{equation}\label{eq3.18}
x_{r+1}<c_{10}x_r^{5^{16}},
\end{equation}
and
\textcolor{white}{xxxxxxxxxxxxxxxxxxxxxxxxxxxxxx}
\begin{displaymath}
x_i\ge c_{10}x_{i-1}^{5^{16}},
\quad
i=1,\ldots,r,
\end{displaymath}
with $x_0=x$.
Furthermore, there exists an integer $k_r$ such that
\begin{equation}\label{eq3.19}
1\le k_r\le\frac{(30)^6}{c_5x_r^5}
\end{equation}
and the $k_r$-th shift under the $\alpha$-flow of the open interval $Q^{(r)}R^{(r)}$ of length $x_r$ splits for the first time,
with the image consisting of a vertex of~$\PPP$, a top interval $w_{j_{r+1}}(0,x_{r+1})$ of length $x_{r+1}$
and a bottom interval of length~$x_{r+1}^*$, where
$x_{r+1}+x_{r+1}^*=x_r$ and \eqref{eq3.18} holds.

We now start with the interval $Q^{(r)}R^{(r)}=w_{j_r}(0,x_r)$ and apply the reverse $\alpha$-flow until it splits for the first time on a vertical edge.
We have the following analog of Claim~1.
The justification is similar to that of Claim~1 in this section or Claim~3 in Section~\ref{sec2}.

\begin{claim3}
Suppose that there is no split among the first $[(30)^6/c_5x_r^5]$ consecutive shifts of the open interval $Q^{(r)}R^{(r)}$ under the reverse
$\alpha$-flow, where $Q^{(r)}R^{(r)}$ has length $x_r>0$, and $c_5=c_5(\alpha)$ is the constant in Lemma~\ref{lem31}.
Let $\LLL_r(t)=\LLL(t+t_r)$ for every $t\le0$.
Then there is a visiting time $t^*$ such that
\begin{displaymath}
0>t^*\ge-\frac{2(30)^6}{c_5x_r^5}
\quad\mbox{and}\quad
\LLL_r(t^*)\in Q^{(r)}R^{(r)},
\end{displaymath}
\textit{i.e.} $\LLL(t^*+t_r)\in Q^{(r)}R^{(r)}$.
Applying the reverse $\alpha$-flow for time $t_r$ then leads to $\LLL(t^*)\in QR$,
so that the conclusion of Lemma~\ref{lem32} holds with suitable constants $c_8=c_8(\alpha)$ and $c_9=c_9(\alpha)$.
\end{claim3}

It is clear that we can assume that the constant $c_5=c_5(\alpha)$ in Lemma~\ref{lem31} satisfies
\begin{displaymath}
0<c_5<1.
\end{displaymath}
We thus make this assumption for the rest of our discussion here.

In view of Claim~3, we may assume that there exists an integer $k_{r+1}$ such that
\begin{equation}\label{eq3.20}
1\le k_{r+1}\le\frac{(30)^6}{c_5x_r^5}
\end{equation}
and the $k_{r+1}$-th shift under the reverse $\alpha$-flow of the open interval $Q^{(r)}R^{(r)}$ of length $x_r$ splits for the first time.

Suppose that the image of the open interval $Q^{(r)}R^{(r)}$ after the first $k_{r+1}$ shifts under the reverse $\alpha$-flow
now consists of a vertex of~$\PPP$, a top interval $w_{j'_{r+1}}(0,y_{r+1})$ of length $y_{r+1}$ and a bottom interval of length $y_{r+1}^*$,
where $y_{r+1}+y_{r+1}^*=x_r$.
Then we delete the bottom interval, keep the top interval and write $Q_{r+1}R_{r+1}=w_{j'_{r+1}}(0,y_{r+1})$.
It then follows from our construction that the geodesic $\LLL(t)$ starting at the point $R$ contains the point~$R_{r+1}$, so that
\begin{displaymath}
R_{r+1}=\LLL(t_{r+1})
\quad\mbox{for some $t_{r+1}$},
\end{displaymath}
where $t_{r+1}$ can be positive or negative.

To estimate $y_{r+1}$ from below, note that the point $R_{r+1}=w_{j'_{r+1}}(y_{r+1})$ is obtained from the point
$R^{(r+1)}=w_{j_{r+1}}(x_{r+1})$ by $k_r$ shifts under the reverse $\alpha$-flow to the point $R^{(r)}$
followed by another $k_{r+1}$ shifts under the reverse $\alpha$-flow from the point~$R^{(r)}$.
Using \eqref{eq3.19} and \eqref{eq3.20}, we see that
\begin{equation}\label{eq3.21}
k_r+k_{r+1}\le\frac{2(30)^6}{c_5x_r^5}\le\frac{2^6(30)^6}{c_5^6x_r^5}=\left(\frac{60}{c_5}\right)^6x_r^{-5}.
\end{equation}
Meanwhile, it follows from \eqref{eq3.2} that there exist integers $n_1,n_2,n_3,n_4$ such that
\begin{displaymath}
y_{r+1}-x_{r+1}=n_1\alpha+n_2\sqrt{2}\alpha-n_3\sqrt{2}-n_4,
\end{displaymath}
with $\vert n_1\vert,\vert n_2\vert,\vert n_3\vert\le k_r+k_{r+1}$ and $n_1^2+n_2^2\ge1$, so that
\begin{displaymath}
y_{r+1}\ge\Vert n_1\alpha+n_2\sqrt{2}\alpha-n_3\sqrt{2}\Vert-x_{r+1}.
\end{displaymath}
It then follows from Lemma~\ref{lem31} and \eqref{eq3.21} that
\begin{align}\label{eq3.22}
y_{r+1}
&
\ge\frac{c_5}{(k_r+k_{r+1})^5}-x_{r+1}
\ge\frac{c_5^{31}}{2^{30}(30)^{30}}x_r^{25}-x_{r+1}
\nonumber
\\
&
\ge\frac{c_5^{36}}{2^{36}(30)^{36}}x_r^{25}
=\left(\frac{c_5}{60}\right)^{36}x_r^{25},
\end{align}
provided that
\textcolor{white}{xxxxxxxxxxxxxxxxxxxxxxxxxxxxxx}
\begin{displaymath}
x_{r+1}<\frac{c_5^{36}}{2^{36}(30)^{36}}x_r^{25}=\left(\frac{c_5}{60}\right)^{36}x_r^{25},
\end{displaymath}
a condition that is clearly satisfied, in view of \eqref{eq3.9} and \eqref{eq3.18}.

We now repeat this argument on the open interval $Q_{r+1}R_{r+1}=w_{j'_{r+1}}(0,y_{r+1})$.

In view of a suitable analog of Claim~3, we may assume that there exists an integer $k_{r+2}$ such that
\textcolor{white}{xxxxxxxxxxxxxxxxxxxxxxxxxxxxxx}
\begin{equation}\label{eq3.23}
1\le k_{r+2}\le\frac{(30)^6}{c_5y_{r+1}^5}
\end{equation}
and the $k_{r+2}$-th shift under the reverse $\alpha$-flow of the open interval $Q_{r+1}R_{r+1}$ of length $y_{r+1}$ splits for the first time.

Suppose that the image of the open interval $Q_{r+1}R_{r+1}$ after the first $k_{r+2}$ shifts
under the reverse $\alpha$-flow now consists of a vertex of~$\PPP$,
a top interval $w_{j'_{r+2}}(0,y_{r+2})$ of length $y_{r+2}$ and a bottom interval of length $y_{r+2}^*$, where $y_{r+2}+y_{r+2}^*=y_{r+1}$.
Then we delete the bottom interval, keep the top interval and write $Q_{r+2}R_{r+2}=w_{j'_{r+2}}(0,y_{r+2})$.
It then follows from our construction that the geodesic $\LLL(t)$ starting at the point $R$ contains the point~$R_{r+2}$, so that
\begin{displaymath}
R_{r+2}=\LLL(t_{r+2})
\quad\mbox{for some $t_{r+2}$},
\end{displaymath}
where $t_{r+2}$ can be positive or negative.

To estimate $y_{r+2}$ from below, note that the point $R_{r+2}=w_{j'_{r+2}}(y_{r+2})$ is obtained from the point
$R^{(r+1)}=w_{j_{r+1}}(x_{r+1})$ by $k_r$ shifts under the reverse $\alpha$-flow to the point $R^{(r)}$
followed by another $k_{r+1}+k_{r+2}$ shifts under the reverse $\alpha$-flow from the point~$R^{(r)}$.
Using \eqref{eq3.21}--\eqref{eq3.23}, we see that
\begin{align}\label{eq3.24}
k_r+k_{r+1}+k_{r+2}
&
\le\frac{2^6(30)^6}{c_5^6x_r^5}+\frac{(30)^6}{c_5y_{r+1}^5}
\le\frac{2^6(30)^6}{c_5^6x_r^5}+\frac{2^{180}(30)^{186}}{c_5^{181}x_r^{125}}
\nonumber
\\
&
\le\frac{2^{216}(30)^{216}}{c_5^{216}x_r^{125}}
=\left(\frac{60}{c_5}\right)^{216}x_r^{-125}.
\end{align}
Meanwhile, it follows from \eqref{eq3.2} that there exist integers $n_1,n_2,n_3,n_4$ such that
\begin{displaymath}
y_{r+2}-x_{r+1}=n_1\alpha+n_2\sqrt{2}\alpha-n_3\sqrt{2}-n_4,
\end{displaymath}
with $\vert n_1\vert,\vert n_2\vert,\vert n_3\vert\le k_r+k_{r+1}+k_{r+2}$ and $n_1^2+n_2^2\ge1$, so that
\begin{displaymath}
y_{r+2}\ge\Vert n_1\alpha+n_2\sqrt{2}\alpha-n_3\sqrt{2}\Vert-x_{r+1}.
\end{displaymath}
It then follows from Lemma~\ref{lem31} and \eqref{eq3.24} that
\begin{align}\label{eq3.25}
y_{r+2}
&
\ge\frac{c_5}{(k_r+k_{r+1}+k_{r+2})^5}-x_{r+1}
\ge\frac{c_5^{1081}}{2^{1080}(30)^{1080}}x_r^{625}-x_{r+1}
\nonumber
\\
&
\ge\frac{c_5^{1296}}{2^{1296}(30)^{1296}}x_r^{625}
=\left(\frac{c_5}{60}\right)^{1296}x_r^{625},
\end{align}
provided that
\textcolor{white}{xxxxxxxxxxxxxxxxxxxxxxxxxxxxxx}
\begin{displaymath}
x_{r+1}<\frac{c_5^{1296}}{2^{1296}(30)^{1296}}x_r^{625}=\left(\frac{c_5}{60}\right)^{1296}x_r^{625},
\end{displaymath}
a condition that is clearly satisfied, in view of \eqref{eq3.9} and \eqref{eq3.18}.

We now repeat this argument on the open interval $Q_{r+2}R_{r+2}=w_{j'_{r+2}}(0,y_{r+2})$.

And so on.

This shift process under the reverse $\alpha$-flow defines a sequence of top intervals
\begin{equation}\label{eq3.26}
Q_{r+i}R_{r+i}=w_{j'_{r+i}}(0,y_{r+i}),
\quad
i\ge1.
\end{equation}
Each interval in \eqref{eq3.26} arises when the $k_{r+i}$-th shift under the reverse $\alpha$-flow
of the open interval $Q_{r+i-1}R_{r+i-1}$ of length $y_{r+i-1}$ splits for the first time, and the integer $k_{r+i}$ satisfies
\textcolor{white}{xxxxxxxxxxxxxxxxxxxxxxxxxxxxxx}
\begin{equation}\label{eq3.27}
1\le k_{r+i}\le\frac{(30)^6}{c_5y_{r+i-1}^5},
\end{equation}
with the convention that $y_r=x_r$.
It is not difficult to prove by induction on $i$ that
\begin{equation}\label{eq3.28}
k_r+k_{r+1}+\ldots+k_{r+i}\le\left(\frac{60}{c_5}\right)^{6^{2i-1}}x_r^{-5^{2i-1}},
\end{equation}
and that
\textcolor{white}{xxxxxxxxxxxxxxxxxxxxxxxxxxxxxx}
\begin{equation}\label{eq3.29}
y_{r+i}\ge\left(\frac{c_5}{60}\right)^{6^{2i}}x_r^{5^{2i}},
\end{equation}
provided that
\textcolor{white}{xxxxxxxxxxxxxxxxxxxxxxxxxxxxxx}
\begin{equation}\label{eq3.30}
x_{r+1}<\left(\frac{c_5}{60}\right)^{6^{2i}}x_r^{5^{2i}},
\end{equation}
a condition that is clearly satisfied, in view of \eqref{eq3.9} and \eqref{eq3.18}, if $i\le8$.
Note that \eqref{eq3.21} and \eqref{eq3.24} represent the inequality \eqref{eq3.28} when $i=1,2$ respectively,
while \eqref{eq3.22} and \eqref{eq3.25} represent the inequality \eqref{eq3.29} when $i=1,2$ respectively.

Furthermore, the geodesic $\LLL(t)$ starting at the point $R$ contains the point~$R_{r+i}$, so that
\textcolor{white}{xxxxxxxxxxxxxxxxxxxxxxxxxxxxxx}
\begin{displaymath}
R_{r+i}=\LLL(t_{r+i})
\quad\mbox{for some $t_{r+i}$},
\end{displaymath}
where $t_{r+i}$ may be positive or negative.

As there are only finitely many vertical edges in the polyrectangle surface~$\PPP$, there will at some point be \textit{edge repetition},
when there exist two integers $i_1$ and $i_2$ satisfying $1\le i_i<i_2$ such that the corresponding top intervals
\begin{displaymath}
Q_{r+i_1}R_{r+i_1}=w_{j'_{r+i_1}}(0,y_{r+i_1})
\quad\mbox{and}\quad
Q_{r+i_2}R_{r+i_2}=w_{j'_{r+i_2}}(0,y_{r+i_2})
\end{displaymath}
lying respectively on the vertical edges $w_{j'_{r+i_1}}$ and~$w_{j'_{r+i_2}}$, overlap.
Thus $j'_{r+i_1}=j'_{r+i_2}$.
Now suppose that $j^*$ is their common value.
Then
\begin{equation}\label{eq3.31}
Q_{r+i_1}R_{r+i_1}=w_{j^*}(0,y_{r+i_1})
\quad\mbox{and}\quad
Q_{r+i_2}R_{r+i_2}=w_{j^*}(0,y_{r+i_2})
\end{equation}
Furthermore, since there are precisely $7$ vertical edges on~$\PPP$, it follows that
\begin{equation}\label{eq3.32}
1\le i_1<i_2\le8.
\end{equation}
This and \eqref{eq3.30} explain our choice of the constant $c_{10}=c_{10}(\alpha)$ given by \eqref{eq3.9},
as well as our choice of the exponent $5^{16}$ in \eqref{eq3.8} and its analogs.

\begin{claim4}
Suppose that there exist integers $i_1$ and $i_2$ satisfying \eqref{eq3.32} such that the following conditions hold:

(1)
For every integer $i$ satisfying $1\le i\le i_2$, there exists an integer $k_{r+i}$ satisfying \eqref{eq3.27}
such that the top interval $Q_{r+i}R_{r+i}$ given by \eqref{eq3.26} arises when the $k_{r+i}$-th shift
under the reverse $\alpha$-flow of the open interval $Q_{r+i-1}R_{r+i-1}$ splits for the first time, where $Q_rR_r=Q^{(r)}R^{(r)}$.

(2)
For every integer $i$ satisfying $1\le i\le i_2$, the condition \eqref{eq3.29} holds.

(3)
There exists an integer $j^*$ such that the condition \eqref{eq3.31} holds.

Then there is a visiting time $t^*$ such that $0<\vert t^*\vert\le\vert t_{r+i_2}-t_{r+i_1}\vert$ and $\LLL(t^*)\in QR$, where
$R_{r+i_1}=\LLL(t_{r+i_1})$ and $R_{r+i_2}=\LLL(t_{r+i_2})$, so the conclusion of Lemma~\ref{lem32} holds
with suitable constants $c_8=c_8(\alpha)$ and $c_9=c_9(\alpha)$.
\end{claim4}

The justification is similar to that of Claim~2 in this section or Claim~4 in Section~\ref{sec2}.

Lemma~\ref{lem32} now follows if we choose the constants $c_8=c_8(\alpha)$ and $c_9=c_9(\alpha)$ appropriately.
\end{proof}

We have the following simple corollary of Lemma~\ref{lem32}.
The proof uses Lemma~\ref{lem31}.

\begin{lem}\label{lem33}
Under the hypotheses of Lemma~\ref{lem32}, the distance between the point $\LLL(t_0)$ and either endpoint $Q$ or $R$ is at least
\begin{displaymath}
c_{11}x^{c_{12}},
\end{displaymath}
where the positive constants $c_{11}=c_{11}(\alpha)$ and $c_{12}=c_{12}(\alpha)$ depend at most on the constant~$c_5=c_5(\alpha)$.
\end{lem}

Let $w$ be a vertical edge of the polyrectangle surface $\PPP$ as shown in Figure~3.3.
As in Section~\ref{sec2}, we can consider a finite segment $\Gamma(\sigma;T)$, given by \eqref{eq2.42}, of the geodesic $\LLL(t)$,
and the maximum gap $\mg(\Gamma(\sigma;T;w))$ defined by \eqref{eq2.44}.
Using this concept, we can establish an extension to Lemma~\ref{lem33} as follows.

\begin{lem}\label{lem34}
Let $\alpha=\root{3}\of{3}/2$.
For any finite segment $\Gamma(\sigma;T)$, given by \eqref{eq2.42}, of a geodesic $\LLL(t)$ of slope~$\alpha$,
let $\mg(\Gamma(\sigma;T;w))=x$ with $0<x<1/2$.
Then the longer finite segment
\begin{displaymath}
\Gamma\left(\sigma;T+\frac{c_8}{x^{c_9}}\right)=\left\{\LLL(t):0\le\vert\sigma-t\vert\le T+\frac{c_8}{x^{c_9}}\right\}
\end{displaymath}
has the property that
\begin{displaymath}
\mg\left(\Gamma\left(\sigma;T+\frac{c_8}{x^{c_9}};w\right)\right)\le x-c_{11}x^{c_{12}}.
\end{displaymath}
\end{lem}

Iterating Lemma~\ref{lem34} sufficiently many times, we obtain the following.

\begin{lem}\label{lem35}
Under the hypotheses of Lemma~\ref{lem34}, there exists positive constants $c_{13}=c_{13}(\alpha)$ and $c_{14}=c_{14}(\alpha)$
such that the longer finite segment
\begin{displaymath}
\Gamma\left(\sigma;T+\frac{c_{13}}{x^{c_{14}}}\right)=\left\{\LLL(t):0\le\vert\sigma-t\vert\le T+\frac{c_{13}}{x^{c_{14}}}\right\}
\end{displaymath}
has the property that
\textcolor{white}{xxxxxxxxxxxxxxxxxxxxxxxxxxxxxx}
\begin{displaymath}
\mg\left(\Gamma\left(\sigma;T+\frac{c_{13}}{x^{c_{14}}};w\right)\right)\le\frac{x}{2}.
\end{displaymath}
\end{lem}

We can now deduce Theorem~\ref{thm2} from Lemma~\ref{lem35} in the same way
as we deduce Theorem~\ref{thm1} from Lemma~\ref{lem25} at the end of Section~\ref{sec2}.

%
%

\section{Extension to algebraic polyrectangles}\label{sec4}

Since the slope $\alpha=\root{3}\of{3}/2$ in Theorem~\ref{thm2} is irrational,
it follows from a celebrated result of Veech that a half-infinite geodesic of slope $\alpha$ on the regular octagon surface
is uniformly distributed on the surface.
However, the result of Veech unfortunately does not say anything about the time-quantitative behavior of such a geodesic.
This is due to the fact that the proof uses ergodic theory.
More precisely, ergodicity of some relevant transformation is established, and then Birkhoff's ergodic theorem is applied.
But Birkhoff's ergodic theorem does not have any explicit error term.

The interesting point of Theorem~\ref{thm2} is that this is a time-quantitative result.
It also has far-reaching generalization to \textit{algebraic polyrectangle surfaces}.
For such a surface, there is a net of the surface on the plane with the following two properties:

(i)
The net is a rectangular polygon, possibly with holes inside.

(ii)
The coordinates of all the vertices are algebraic numbers.

We call such a net an \textit{algebraic net} of the surface.

We can study geodesic flow on an arbitrary algebraic polyrectangle surface.
Using the standard \textit{$4$-copy construction} trick, we can reduce the problem if necessary
to one concerning $1$-direction geodesic flow on an arbitrary algebraic polyrectangle translation surface.
With a straightforward generalization of the proof of Theorem~\ref{thm2}, one can establish the following result.

\begin{thm}\label{thm3}
Let $\PPP$ be an algebraic polyrectangle translation surface, and let $\LLL(t)$ be a $1$-direction geodesic on $\PPP$
such that the slope $\alpha$ is an algebraic number.
Let $K$ denote the smallest extension of $\Qq$ that contains the coordinates of all the vertices of an algebraic net of~$\PPP$,
and suppose that $\alpha\not\in K$.
Then there are explicitly computable positive constants $c_{15}=c_{15}(\PPP;\alpha)$ and $c_{16}=c_{16}(\PPP;\alpha)$
such that, for any integer $n\ge2$ and any aligned square $A$ of side length $1/n$ on the surface $\PPP$,
there exists a real number $t_0$ such that
\begin{displaymath}
0\le t_0\le c_{15}n^{c_{16}}
\quad\mbox{and}\quad
\LLL(t_0)\in A.
\end{displaymath}
\end{thm}

In view of unfolding, the same holds for any billiard on an algebraic polyrectangle table~$\PPP$,
where the initial slope is an algebraic number independent of the number field generated by~$\PPP$.

Theorem~\ref{thm2} is clearly a special case of Theorem~\ref{thm3}.
Suppose that in Figure~3.2, the vertex $A_6$ has coordinates $(0,0)$, and the vertex $A_5$ has coordinates $(0,\sqrt{2})$.
Then the coordinates of all the vertices of this algebraic net in Figure~3.2 are all in the quadratic field $\Qq(\sqrt{2})$,
and it is clear that the cubic slope $\alpha=\root{3}\of{3}/2$ is not in this number field.

The class of \textit{algebraic polyrectangle translation surfaces} is large.
For instance, every regular polygon billiard is equivalent to a $1$-direction geodesic flow on one of these surfaces.
Note that the same holds for any Veech surface.

It is interesting to mention that for the family of \textit{street-rational polyrectangle surfaces}, we can prove stronger results,
at least for an explicit countable set of slopes.
Here we can establish superdensity and also determine the \textit{irregularity exponent}
which represents a precise form of time-quantitative uniformity; see \cite{BDY1,BDY2,BCY1,BCY2}.
Note that this family also contains every regular polygon billiard surface and every Veech surface.

These are the two general classes of flat surfaces for which we can prove time-quantitative results
concerning the long-term behavior of geodesic flow.
These will be discussed in our first monograph in progress.

\end{document}